\documentclass[a4paper,11pt]{article}

\usepackage{amsmath}
\usepackage{amssymb}
\usepackage{stmaryrd}
\usepackage{amsthm}
\usepackage{mathrsfs}
\usepackage{latexsym}
\usepackage{amssymb}
\usepackage{amscd}
\usepackage{tikz}
\usepackage{enumerate}
\usepackage[colorlinks=true]{hyperref}
\usepackage{tikz-cd}
\usepackage{adjustbox}
\usepackage{enumitem}
\usepackage{fullpage}
\usepackage{mathabx}
\usepackage{caption}
\usepackage{subcaption}
\usepackage{palatino}
\usepackage{authblk}
\usepackage{float}

\hypersetup{colorlinks=true,linkcolor=[rgb]{0.5, 0.0, 0},citecolor=[rgb]{0.5, 0.0, 0.5}, filecolor=magenta,urlcolor=blue}

\numberwithin{equation}{section}
\newtheorem{theorem}{Theorem}[section]
\newtheorem{proposition}[theorem]{Proposition}
\newtheorem{lemma}[theorem]{Lemma}

\newtheorem{corollary}[theorem]{Corollary}
\theoremstyle{definition}
\newtheorem{rmk}[theorem]{Remark}
\newtheorem{definition}[theorem]{Def{i}nition}
\newtheorem{example}[theorem]{Example}

\newcommand{\tn}{\mathrm}

\usepackage{hyperref}

    \DeclareMathOperator{\Top}{top_{\mkern2mu}}
    
    \newcommand{\op}{\operatorname{op}}
    \newcommand{\enveloping}[1]{{#1}^{\operatorname{e}}}
    \newcommand{\quotient}[2]{#1 /\!/ #2}
    
    \newcommand{\pd}{\operatorname{pd}}
    \newcommand{\Hom}{\operatorname{Hom}}
    \newcommand{\Tor}{\operatorname{Tor}}
    \newcommand{\Ext}{\operatorname{Ext}}
    
    \DeclareMathOperator{\torsion}{\operatorname{Tor}}
    \DeclareMathOperator{\totalComplex}{\operatorname{Tot}}
    
    \DeclareMathOperator{\gldim}{\operatorname{gldim}}
    
    \newcommand{\cat}[1]{\textnormal{#1}}
    \newcommand{\catmod}[1]{#1\text{-}\cat{mod}}

    \newcommand{\triangleHeight}{2.5*0.86}
    \newcommand{\triangleSize}{2.5*1}

\title{Relative Global Dimension of Controllable Extensions}

\author{Kostiantyn Iusenko, Roger R. Primolan*}
\affil{Departamento de Matemática, Univ. de São Paulo, Rua do Matão, 1010, São Paulo, SP, CEP: 05508-090 -- Brazil.}
        
\date{\today}

\begin{document}

\maketitle

\begin{abstract}
  
  We study relative global dimensions of extensions \(B \subseteq A\) of Artin algebras over a perfect field. A controllable extension is introduced as one for which the relative global dimension is determined by the ordinary global dimension of the quotient \(A/AJ(B)A\). General inequalities and sufficient conditions for controllability are established, including the case in which \(J(B)\) is a two-sided ideal of \(A\).
  
  The interaction between relative global dimensions and relative projective dimensions of the regular bimodule is then studied. The class of homologically controllable extensions is introduced, and it is shown to be closed under tensor products. Consequently, for homologically controllable extensions \(B \subseteq A\) and \(D \subseteq C\), we obtain 
    \[
        \operatorname{gldim}(A \otimes C, B \otimes D)
          =
          \operatorname{gldim}(A,B) + \operatorname{gldim}(C,D).
    \]
  The proof relies on a relative Künneth-type result for tensor products of complexes, which may be of independent interest.
  
  Finally, controllable extensions with prescribed relative global dimension are constructed, along with families of non-controllable extensions for which the relative global dimension differs arbitrarily from the global dimension of \(A/AJ(B)A\).

  \medskip

\end{abstract}
  
\let\thefootnote\relax\footnote{\textit{Email addresses:} iusenko@ime.usp.br (Kostiantyn Iusenko), roger.primolan@alumni.usp.br (Roger R. Primolan, *corresponding author).}
\section{Introduction}

Relative homological algebra, initiated by Hochschild \cite{Hoc56}, provides a
framework for studying extensions of algebras
\(B\subseteq A\) through homological invariants that interpolate between the
homological theories of \(A\) and \(B\). It has received considerably less
attention in the setting of associative algebras than its counterpart in the
representation theory of finite groups. In the latter setting, relative
homological methods have had a profound impact on modular representation
theory and block theory; see, for instance, \cite{Lin18}. By contrast, many
basic questions concerning relative homological dimensions of algebra
extensions remain poorly understood.

Among the central invariants of an extension \(B\subseteq A\) is the relative
global dimension \(\gldim(A,B)\). It interpolates between the homological
dimensions of the two algebras: when \(B=\Bbbk\), it recovers the ordinary
global dimension of \(A\), whereas \(\gldim(A,A)=0\). Early work focused
primarily on the case \(\gldim(A,B)=0\), see \cite{Hir59, HS66, Gre75}, and several
classes of extensions satisfying this property have been identified.
Nevertheless, determining relative global dimensions and understanding which
algebraic features of an extension govern them remain largely open problems.

More recently, relative homological algebra has found applications to
homological conjectures for finite-dimensional algebras, such as the
finitistic dimension conjecture and Han's conjecture; see
\cite{XX13,CLMS22,IM25} and the references therein. In these applications,
the finiteness of the relative global dimension plays a fundamental role in
transferring homological properties between an algebra and a subalgebra.
In a different direction, \cite{IMP25} establishes a relative analogue of the
Auslander--Buchsbaum--Serre theorem, showing that the finiteness of a suitable
relative global dimension characterizes the smoothness of extensions of
commutative Noetherian rings.

The present work is motivated by the problem of understanding to what extent
relative global dimensions can be described in terms of ordinary homological
properties of quotient algebras. More precisely, given an extension
\(B\subseteq A\), we investigate when the quotient
\( A/AJ(B)A \)
controls the relative global dimension. We call such extensions
\emph{controllable}; that is, those satisfying
\[
\gldim(A,B)=\gldim(A/AJ(B)A).
\]
This point of view transfers questions from relative homological algebra to
the ordinary homological theory of Artin algebras.

The first part of the paper develops certain tools for studying controllability. In particular, we establish a general inequality
\(
\gldim(A/AJ(B)A)\leq \gldim(A,B)
\)
and derive several sufficient conditions for equality. Our main result shows
that every extension satisfying
\(J(B)\) is a two-sided ideal in \(A\) is controllable (Theorem \ref{teo:BilateralIdealIsControllable}). Consequently, the relative global dimension of such an
extension is completely determined by the ordinary global dimension of the
quotient algebra \(A/J(B)\).

The second part concerns the relationship between relative global dimensions
and relative projective dimensions of the regular bimodule. We study
controllable extensions satisfying
\( \gldim(A,B)=\pd_{(A^{e},B^{e})}A, \)
which we call \emph{homologically controllable}. This class is closed under
tensor products, yielding formulas analogous to the classical additivity of
global dimension. In particular, if \(B\subseteq A\) and \(D\subseteq C\) are
homologically controllable extensions, then Theorem
\ref{teo:ComputesGlobalDimensionOfTensorProductOfExtensions} gives
\[
  \gldim(A\otimes C,B\otimes D)
  =
  \gldim(A,B)+\gldim(C,D).
\]
In the commutative setting, \cite[Proposition 3.21]{IMP25} establishes the
same formula under mild hypotheses on the extensions. The theorem above therefore extends this additivity result to a broader class
of algebra extensions, including noncommutative ones. A key ingredient in the proof is a relative Künneth-type theorem for tensor
products of complexes; see Theorem~\ref{teo:RelativeKunnethFormula}. More
precisely, we construct explicit homotopies for total complexes arising from
tensor products of complexes admitting relative homotopies. This yields a
relative Künneth formula for Tor groups and provides the homological machinery
needed to establish the additivity of relative global dimensions under tensor
products. We believe that these results may also be of independent interest.

The developed approach also yields constructions of non-trivial extensions
with prescribed relative global dimensions. In particular, every
finite-dimensional algebra \(C\) can be realized as a quotient \(A/J(B)\) of a
controllable extension \(B\subseteq A\) satisfying
\(
\gldim(A,B)=\gldim(C).
\)
We further discuss trivially twisted extensions and compare our methods with
existing constructions from \cite{Guo18}. Finally, we show that the class of
controllable extensions is not exhaustive by constructing families of
non-controllable extensions exhibiting arbitrarily large gaps between
\(\gldim(A,B)\) and \(\gldim(A/AJ(B)A)\).

The paper is organized as follows. Section~\ref{sec:Premilinaries} recalls the necessary background
from relative homological algebra. Section~\ref{sec:ControllableExtensions} introduces controllable
extensions and develops sufficient criteria for controllability. Section~\ref{sec:TensorProductOfExtensions}
studies homologically controllable extensions and their behavior under tensor products. The
final Section~\ref{sec:NonControllableExtensions} is devoted to examples and to the construction of
non-controllable extensions.

\medskip

\textbf{Use of AI tools.} Generative AI tools were used solely for language editing of selected sentences. They were not used for any mathematical work, including the formulation of results, proofs, calculations, or verification. The authors reviewed the resulting text and are fully responsible for the contents of the paper.

\medskip

\textbf{Acknowledgements.} We thank Eduardo N. Marcos for stimulating discussions and for valuable feedback on an earlier draft of this manuscript. We also thank Hipolito Treffinger for posing the question that motivated the investigation in Section~\ref{sec:NonControllableExtensions} as well as John MacQuarrie for stimulating discussions and his help in constructing Example~\ref{ex:NonControllableExtension}. K.I. was partially supported by FAPESP grant 2018/23690-6, CNPq Universal Grant 405540/2023-0, and FAPEMIG Project APQ-03491-25. R.R.P. was partially supported by FAPESP grant 2025/19298-7 and CAPES--Finance Code 001 during their master's dissertation and ongoing PhD thesis.

\section{Preliminaries} \label{sec:Premilinaries}

Throughout this work, \(\Bbbk\) denotes a \emph{perfect field}, and all algebras are Artin algebras over \(\Bbbk\). For an algebra \(A\), we write \(\catmod{A}\) for the category of finitely generated \emph{left} \(A\)-modules, and \(J(A)\) for the \emph{Jacobson radical} of \(A\). For a module \(M \in \catmod{A}\), its \emph{radical} is \(\operatorname{rad}_{A}(M) = J(A)M\), and its \emph{top} is \(\Top_{A}(M) \doteq M / \operatorname{rad}_{A}(M)\). \emph{Right} \(A\)-modules are identified with left modules over the opposite algebra \(\catmod{A^{\op}}\), and, as usual, for any algebra \(C\), the category of \(A\)--\(C\) \emph{bimodules} is identified with \(\catmod{A \otimes_{\Bbbk} C^{\op}}\); in particular, \(A\)--\(A\) bimodules correspond to left modules over the \emph{enveloping algebra} \(\enveloping{A} \doteq A \otimes_{\Bbbk} A^{\op}\).

Relative homological algebra (cf. \cite{Hoc56}) is based on the concept of relatively projective modules, which we now recall. Given an extension of \(\Bbbk\)-algebras \(B \subseteq A\), an \(A\)-module \(M\) is said to be \((A,B)\)\emph{-projective}, or \emph{relatively projective}, if it satisfies one -- and hence all -- of the following equivalent conditions (see \cite[Section 1]{Hoc56} for details):
    \begin{itemize}
        \item[(i)] the multiplication map $\mu_M : A \otimes_B M \to  M$ is a split epimorphism of $A$-modules, where $A\otimes_B M$ carries the natural left $A$-module structure;
        \item[(ii)] \(M\) is isomorphic to a direct summand of the induced module \(A \otimes_B X\), for some \(B\)-module \(X\);
        \item[(iii)] if an \(A\)-module homomorphism onto \(M\) admits a section in the category of \(B\)-modules, then it also admits a section in the category of \(A\)-modules.
    \end{itemize}
 When \(B = \Bbbk\), the \((A,\Bbbk)\)-projective modules are the usual projective \(A\)-modules. On the other end of the spectrum, when \(B = A\), every \(A\)-module is relatively projective. Following the notation of \cite[Page 7]{XX13}, the \emph{full subcategory} of \(\catmod{A}\) defined by all \((A,B)\)-projective modules is denoted by \(\mathcal{P}(A, B)\).

A \emph{complex} of \(A\)-modules 

    \[
        \begin{tikzcd}
            \cdots \arrow[r] & M_{n+1} \arrow[r, "d_{n+1}"] & M_n \arrow[r, "d_n"] & M_{n-1} \arrow[r] & \cdots,
        \end{tikzcd}
    \]
\emph{admits a} \(B\)-\emph{homotopy} if there are maps of \(B\)-modules \(h_n : M_n \to M_{n+1}\) satisfying
    
    \[
        d_{n+1} \circ h_n + h_{n-1} \circ d_n = 1_{M_n},   
    \]

\noindent for all indices. An \emph{\((A, B)\)-exact sequence} or \emph{relative exact sequence} is a complex of \(A\)-modules that admits a \(B\)-homotopy. This is equivalent to the following statements:

    \begin{itemize}
        \item there exists maps of \(B\)-modules \(h_n: M_n \to M_{n+1}\) such that \(d_{n+1} \circ h_n \circ d_{n+1} = d_{n+1}\), for all indices, and;
        \item each \(\ker(d_n)\) is a direct summand of \(M_n\) as a \(B\)-module.
    \end{itemize}

Clearly, when \(B\) is semisimple, every exact sequence admits a \(B\)-homotopy -- therefore any exact sequence is \((A,B)\)-exact in this case. On other hand, when \(B = A\), the \((A,A)\)-exact sequences are precisely the split exact sequences of \(A\)-modules. 

Given and \(A\) module \(M\), an \((A, B)\)-\emph{projective resolution} (or \emph{relative resolution}) of \(M\)  is an \((A, B)\)-exact sequence 
    \[
        \begin{tikzcd}
            \cdots \arrow[r] & P_{n+1} \arrow[r, "d_{n+1}"] & P_n \arrow[r] & \cdots \arrow[r] & P_0 \arrow[r, "d_0"] & M \arrow[r] & 0,
        \end{tikzcd}
    \]

\noindent such that each \(P_n\) is relatively projective, the presented definition is equivalent to the one in \cite[Page 7]{XX13}. It is known that every module admits a relative resolution, see \cite[Section 2]{Hoc56}. The \emph{deleted} version of the above relative resolution is the subcomplex replacing \(M\) with \(0\), such complex we will denote by \(P_{\bullet}\). The \((A, B)\)-\emph{projective dimension} or \emph{relative projective dimension} of \(M \in \catmod{A}\), denoted by \(\pd_{(A,B)} M\), is the least length of a deleted relative resolution for \(M\) or infinity. The \emph{relative global dimension} of the extension \(B \subseteq A\) is defined by
    \[
        \gldim(A, B) = \sup\{ \pd_{(A,B)} M \mid M \in \catmod{A} \}.
    \]

For \(M, N \in \catmod{A}\), the \emph{relative Ext-functors} are defined as
\[
\tn{Ext}^n_{(A,B)}(M, N) = H^n(\Hom_A(P_\bullet, N)),
\]
where \(P_\bullet\) is a deleted \((A,B)\)-projective resolution of \(M\). The \emph{relative Tor-functors} are defined analogously. For further details, see \cite[Section~2]{Hoc56}.

Just as in classical situation, it holds 
\begin{equation} \label{eq:PdAsTorZero}
  \pd_{(A, B)} M = \sup \{n:\tn{Tor}^{(A,B)}_n (-, M) \neq 0\}
\end{equation}
and
\begin{equation} \label{eq:GldimOpposite}
  \gldim(A, B) = \gldim(A^{\op}, B^{\op}).
\end{equation}
see \cite{IMP26} for more details. 

\medskip

\textbf{Conventions and Notations:} our objective is to compare relative homological dimensions to their classical counterparts, for that reason we use well known results from the literature that hold whenever the algebra is over a perfect field. Even though some of our results are valid for Artin algebras over an arbitrary commutative Artinian ring \(R\), we will always assume that the algebras are Artin algebras over a perfect field \(\Bbbk\) to have clearer statements. 
To simplify notation, especially in subscripts, we write \(\quotient{A}{B} \doteq A/AJ(B)A\) for an extension \(B \subseteq A\), and follow the usual convention \(\otimes \doteq \otimes_{\Bbbk}\).
\section{Controllable Extensions} \label{sec:ControllableExtensions}

For a Artin \(\Bbbk\)-algebra \(A\), computing \(\gldim (A)\) reduces to calculating \(\pd_A S\) for each simple \(A\)-module \(S\). Thus, \(\gldim (A)\) can be determined by examining only a finite amount of \(A\)-modules. In fact, it holds that
  \[
    \gldim A = \pd_A \frac{A}{J(A)}.
  \]
However, in the case of relative global dimension, we must in principle compute \(\pd_{(A,B)} M\) for all indecomposable \(A\)-modules \(M\). This naturally raises the following question: under which conditions on the extension \(B \subseteq A\) does the relative global dimension coincide with the usual global dimension of a suitable quotient of \(A\), in which the algebra \(B\) is, to some extend, ``ignored"? Such equality would allow to control the relative global dimension     using non-relative (co)homology and it motivates the following definition.

An extension \(B \subseteq A\) is said to be \emph{controllable} if
  \begin{equation} \label{def:ControllableExtensionsAndInequality}
      \gldim(A,B) = \gldim\left(\frac{A}{AJ(B)A}\right) = \gldim\left( \quotient{A}{B} \right).
  \end{equation}
As a weaker requirement, the extension $B \subseteq A$ is said to satisfy the \emph{controllable inequality} if
  \[
    \gldim(A,B) \geqslant  \gldim\left(\frac{A}{AJ(B)A}\right) = \gldim\left( \quotient{A}{B} \right).
  \]
Of course, the extensions \(\Bbbk \subseteq A\) and \(A \subseteq A\) are controllable. Moreover, notice that by equation \ref{eq:GldimOpposite}, if \(B \subseteq A\) is controllable, then so it is the \emph{opposite extension} \(B^{\op} \subseteq A^{\op}\).

\subsection{Basic Examples}

\subsubsection{Semisimple Extensions}

\begin{lemma}\label{lemma:IdealGenJBEqualsJA}
  Let \(B \subseteq A\) be an extension of Artinian algebras such that \(A/J(A) \cong B/J(B)\) as \(\enveloping{B}\)-modules. Then \(AJ(B)A = J(A)\) if and only if \(J(B)=J(A)\).
\end{lemma}
\begin{proof}
Assume \(AJ(B)A=J(A)\), then it holds \(J(B) \subseteq J(A)\) that, together with the isomorphism \(A/J(A)\cong B/J(B)\), implies
\[ A=B+J(A), \qquad B\cap J(A)=J(B). \]
Then
\[ J(A)=(B+J(A))J(B)(B+J(A)) \subseteq J(B)+J(A)J(B)+J(B)J(A)+J(A)J(B)J(A). \]
Since \(J(B)\subseteq J(A)\), it follows that
\( J(A)\subseteq J(B)+J(A)^2. \)
The reverse inclusion being obvious, we obtain
\[ J(A)=J(B)+J(A)^2. \]
By induction, \(J(A)=J(B)+J(A)^n\)
for all \(n\ge2\). Since \(A\) is Artinian, \(J(A)\) is nilpotent, say \(J(A)^N=0\), and therefore
\[
J(A)=J(B)+J(A)^N=J(B).
\]
The converse is immediate.
\end{proof}

\begin{corollary}\label{cor:EqualRadicalExtensionSemisimple}
  Let \(B \subseteq A\) be a controllable extension of algebras such that \(A/J(A) \cong B/J(B)\) as \(\enveloping{B}\)-modules. Then \(\gldim(A, B)=0\) if and only if \(J(B)=J(A)\).
\end{corollary}
\begin{proof}
  If \(J(B) = J(A)\), then \(\gldim(A, B)=0\) by \cite[Lemma~2.11]{XX13}. Conversely, since the extension is controllable \(\gldim(\quotient{A}{B})=0\) and, therefore, \(AJ(B)A = J(A)\).
\end{proof}

\begin{rmk}
The study of relative global dimensions was initiated in the case \(\gldim(A,B)=0\), see \cite{Hir59, HS66, Gre75}, but a complete characterization of such extensions remains open. The above result shows that, within a sufficiently large class of extensions, controllability with \(\gldim(A,B)=0\) is characterized by \(J(B)=J(A)\).
\end{rmk}

\subsubsection{Trivial Extensions}\label{ex:TrivialExtensions}

Let \(B\) be a finite-dimensional \(\Bbbk\)-algebra and \(M\) a finitely generated \(B\)-bimodule. We write \(T[B,M]\) for the tensor algebra of \(M\) over \(B\) and \(B\ltimes M\) for the corresponding trivial extension. Both come equipped with a natural inclusion \(B\subseteq A\).
\sloppy When \(M\) is \emph{\(2\)-tensor-nilpotent}, meaning that \(M \otimes_{B} M = 0\), we have \(T[B,M] \cong B \ltimes M\). In this case, it is possible to compute relative dimensions for the trivial extension using the theory developed for tensor extensions. Recall, from \cite[Theorem~4.1.1]{Pri23}, that \(\gldim(T[B, M], B) \leqslant 1\) with equality holding if and only if \(M \neq 0\).

\begin{proposition}\label{prop:GlobalDimensionTrivialExtensions2Nilpotent}
  Let \(B\) be a finite-dimensional \(\Bbbk\)-algebra and let \(M \neq 0\) be a \(2\)-nilpotent \(B\)-bimodule. Then the extension \(B \subseteq A = B \ltimes M\) is controllable and 
    \[
      \gldim(A,B) = 1.
    \]        
\end{proposition}

\begin{proof}
Since \(\gldim(A, B)=1\), it suffices to prove that
\( \gldim(\quotient{A}{B})=1.\)
Notice
\[AJ(B)A=J(B)\oplus J_{B^e}(M),\]
we obtain
\[
\quotient{A}{B} \cong \frac{B}{J(B)} \ltimes \Top_{B^e}M\cong T\!\left[\frac{B}{J(B)},\,\Top_{B^e}M\right],
\]
where the latter isomorphism follows from
\[
\Top_{B^e}M\otimes_{B/J(B)}\Top_{B^e}M=0.
\]
Since \(B/J(B)\) is semisimple, the algebra on the right is hereditary and non-semisimple.
\end{proof}

Trivial extensions by \(2\)-nilpotent modules appear in \cite[Theorem~A]{GPS21}, where the authors study reduction procedures for algebras in such way that they preserve the finiteness of the finitistic dimension. One such technique, known as \emph{arrow removal}, consists in constructing a trivial extension \(\Lambda = \Gamma \ltimes P\), where \(P\) is a projective \(\Gamma^{e}\)-module satisfying \(P \otimes_{\Gamma} P = 0\). Hence, all such extensions are controllable. This is not the first appearance of arrow removal in the relative setting: the authors of \cite{IM25} generalized \cite[Theorem~A]{GPS21}, see \cite[Remark~6.13]{IM25}.

\begin{example}
  Below we provide an example of a trivial extension by a $2$-nilpotent module $M$ that is not $B^{e}$-projective -- that is, an example not satisfying the hypotheses of \cite[Theorem~A]{GPS21}. Let
    \[
      \begin{tikzcd}
        Q: 1 \arrow[r, "\alpha"] & 2 \arrow[r, "\beta", shift left] \arrow[r, "\gamma"', shift right] & 3,
      \end{tikzcd}
    \]
  and consider \(B = \Bbbk Q\). Let \(M \in \catmod{B^{e}}\) be the indecomposable module of dimension two supported on the vertices \(e_{2} \otimes e_{1}\) and \(e_{3} \otimes e_{1}\), such that \(\beta \otimes e_{1} \colon (e_{2} \otimes e_{1})M \to (e_{3} \otimes e_{1})M\) is the zero map, while \(\gamma \otimes e_{1} \colon (e_{2} \otimes e_{1})M \to (e_{3} \otimes e_{1})M\) is bijective. Then \(M\) is not a \(B^{e}\)-projective module and satisfies \(M \otimes_{B} M = 0\).
  Consequently, the extension \(B \subseteq B \ltimes M\) is controllable.
\end{example}

\begin{rmk}
  Another application of trivial extensions by \( 2 \)-nilpotent modules arises from triangular algebras. Let \( C \) and \( D \) be finite-dimensional algebras and \( M \) be a \( (D,C) \)-bimodule. Set \( B = C \oplus D \), consider the triangular algebra
    \[
      A =
        \begin{pmatrix}
          C & 0 \\
          M & D
        \end{pmatrix}
      \enspace \textnormal{and denote} \enspace
      \widehat{M} =
        \begin{pmatrix}
          0 & 0 \\
          M & 0
        \end{pmatrix}.
    \]
  Then \( A \cong B \ltimes \widehat{M} \) and \( \widehat{M} \) is a \( 2 \)-nilpotent \( B \)-bimodule. Consequently, the extension \( B \subseteq A \) is controllable.
\end{rmk}
    
\subsection{Sufficient Condition for Controllable Inequality} \label{subsec:CategoricalControllableInequallity}

This subsection focus on a set of sufficient conditions to obtain the controllable inequality. Recall that for an ideal $I \lhd A$, the canonical projection $A \to A/I$ induces a functor
  \[
    F: \catmod{A/I} \to \catmod{A}.
  \]
When \(I=AJ(B)A\), there is a transfer of homological properties from \(\quotient{A}{B}\) to their \((A, B)\)-relative counterparts.

\begin{lemma}\label{lemma:ControllableABExact}
  If \(f \colon M \to N\) is a surjective homomorphism of \(\quotient{A}{B}\)-modules, then, when regarded as a homomorphism of \(A\)-modules, it admits a \(B\)-linear section.
\end{lemma}
    
\begin{proof}
  Every \(\quotient{A}{B}\)-module is \(B\)-semisimple, hence any epimorphism between them admits a section in \(\catmod{B}\).
\end{proof}

\begin{rmk}\label{rmk:QuocientCondition}
  The following results will assume that \({}_A \quotient{A}{B} \in \mathcal{P}(A,B)\). Note that if \(AJ(B) \lhd A\), then \(A \otimes_B (B/J(B)) \cong \quotient{A}{B}\) is \((A, B)\)-projective, since relative projective modules are direct summands of induced modules. Therefore the results will hold if \(AJ(B) \lhd A\), and this hypothesis will be used in later results.
\end{rmk}

\begin{theorem}\label{teo:HomologicalConditionControllableInequality}        
  Let \(B \subseteq A\) be an extension of \(\Bbbk\)-algebras such that \({}_A \quotient{A}{B} \in \mathcal{P}(A,B)\). Then there are natural isomorphisms
    \[
      \tn{Ext}^n_{(A,B)}(M,N) \cong \Ext_{\quotient{A}{B}}^n(M,N)
        \enspace \textnormal{and} \enspace
      \tn{Tor}^{(A,B)}_n(L,M) \cong \torsion^{\quotient{A}{B}}_n(L,M),
    \]
  for all \(M, N \in \catmod{\quotient{A}{B}}\), \(L \in \catmod{(\quotient{A}{B})^{\op}}\), and \(n \in \mathbb{N}\).
\end{theorem}
    
\begin{proof}
  Let \(M\) be a \(\quotient{A}{B}\)-module, and consider a \(\quotient{A}{B}\)-projective resolution
    \[
      \begin{tikzcd}
        \cdots \arrow[r] & P_2 \arrow[r, "d_2"] & P_1 \arrow[r, "d_1"] & P_0 \arrow[r, "d_0"] & M \arrow[r] & 0.
      \end{tikzcd}
    \]
  Since \(\operatorname{add}_{A}(\quotient{A}{B}) = \operatorname{add}_{\quotient{A}{B}}(\quotient{A}{B}) \subseteq \mathcal{P}(A,B)\),
  Lemma~\ref{lemma:ControllableABExact} implies that the above exact sequence is also a relative \((A,B)\)-projective resolution. Consequently, it can be used to compute both \(\quotient{A}{B}\)-homological groups and \((A,B)\)-relative homological groups via the usual (co)homological procedures. The resulting complexes are isomorphic, due to the natural isomorphisms
    \[
      \Hom_{\quotient{A}{B}}(X,Y) \cong \Hom_{A}(X,Y)
        \enspace \text{and} \enspace
      Z \otimes_{\quotient{A}{B}} X \cong Z \otimes_{A} X,
    \]
  where \(X,Y \in \catmod{\quotient{A}{B}}\) and \(Z \in \catmod{(\quotient{A}{B})^{\op}}\).
\end{proof}
    
\begin{corollary}\label{cor:CategoricalControllableInequality}
  Let \(B\subseteq A\) be an extension of \(\Bbbk\)-algebras such that \(_{A}\quotient{A}{B} \in \mathcal{P}(A,B)\). Then:
    \begin{enumerate}
      \item \(\pd_{\quotient{A}{B}}M=\pd_{(A,B)}M\), for any \( \quotient{A}{B} \)-module \(M\);
      \item \(\gldim\left(\quotient{A}{B}\right)\leqslant \gldim(A,B)\), and;
      \item If \(\gldim\left(\quotient{A}{B}\right)\) is infinite, then \(B\subseteq A\) is controllable. 
    \end{enumerate}
\end{corollary}

\begin{proof}
  It suffices to prove the first assertion, for this fix \(M \in \catmod{\quotient{A}{B}}\). Since every \(\quotient{A}{B}\)-projective resolution is also an \((A,B)\)-projective resolution, we have
    \begin{align*}
      \pd_{(A, B)} M &= \min\{\ell (P_{\bullet}) \ : \ P_{\bullet} \ \textnormal{is an} \ (A, B)\textnormal{-projective resolution}\} \\
        & \leqslant \min\{\ell (P_{\bullet}) \ : \ P_{\bullet} \ \textnormal{is an} \ \quotient{A}{B}\textnormal{-projective resolution}\} = \pd_{\quotient{A}{B}} M
    \end{align*}
  For the other inequality, observe that
    \begin{align*}
      \pd_{\quotient{A}{B}} M &= \sup\{n: \torsion^{\quotient{A}{B}}_n(-, M) \neq 0 \} \\
        & \leqslant \sup\{n: \torsion^{(A, B)}_n(-, M) \neq 0\} \leqslant \pd_{(A, B)} M.
    \end{align*}
\end{proof}

\subsection{Applications}

\subsubsection{Same Radical Quotient}

Let \(B \subseteq A\) be an extension such that \( A/J(A) \cong B/J(B) \) as \(\enveloping{B}\)-modules and \(AJ(B) \lhd A\). Recall that \(\quotient{A}{B} \doteq A/AJ(B)A\) is always an algebra, so we can impose to it usual qualities of said objects.

\begin{proposition}\label{prop:BinaryClassicalGldim}
  Suppose that \(\gldim\left( \quotient{A}{B} \right) \in \{0, \infty\}\). Then \(B \subseteq A\) is controllable.
\end{proposition}
\begin{proof}
  It follows directly from Corollaries \ref{cor:EqualRadicalExtensionSemisimple} and \ref{cor:CategoricalControllableInequality}.
\end{proof}

\begin{corollary}
  The extension \(B \subseteq A\) is controllable whenever one of the following conditions holds:
  \begin{enumerate}
    \item[(i)] Both \(A\) and \(B\) are local algebras, that is, \(A/J(A)\cong B/J(B)\cong \Bbbk.\)
    \item[(ii)] The algebra \(\quotient{A}{B}\) is self-injective in the usual sense.
    \item[(iii)] Relative projective and relative injective modules coincide.
  \end{enumerate}
\end{corollary}
\begin{proof}
  By Proposition \ref{prop:BinaryClassicalGldim}, it suffices to show that
  \(\gldim(\quotient{A}{B})\in\{0,\infty\}\) in each case.
  \begin{enumerate}
    \item[(i)] This follows from \cite[Propositions~14 and~15(a)]{Aus55}.

    \item[(ii)] Since \(\quotient{A}{B}\) is self-injective, its global dimension is either \(0\) or \(\infty\).

    \item[(iii)] Since \(AJ(B)\lhd A\), the \(A\)-module \(\quotient{A}{B}\) is \((A,B)\)-projective, and hence is relatively injective by hypothesis. Moreover, every \(\quotient{A}{B}\)-module is \(B\)-semisimple when viewed as an \(A\)-module. Therefore every short exact sequence of \(\quotient{A}{B}\)-modules is \((A,B)\)-split, and it follows that the regular module \(\quotient{A}{B}\) is injective over itself. Thus \(\quotient{A}{B}\) is self-injective, and the claim follows from item~(ii).\end{enumerate}\end{proof}

\begin{example}
  The following examples illustrate the above corollary:
  \begin{enumerate}
    \item The extension
      \[
        B=\Bbbk[z]/\langle z^2\rangle
        \subseteq
        A=B[x,y]/\langle x^2,y^2,xy\rangle
      \]
      is local, commutative, and therefore controllable; neither \(A\) nor \(\quotient{A}{B}\) are self-injective.

    \item Consider the quiver \(Q\)
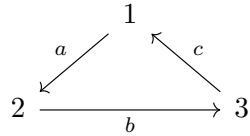
      \[
        \begin{tikzcd}
          {} & {1} \arrow[ld, "a"'] & {} \\
          {2} \arrow[rr, "b"'] & {} & {3} \arrow[lu, "c"']
        \end{tikzcd}
      \]
      and let \(A=\Bbbk Q/I\), where \(I=\langle ba,cb\rangle\). If \(B\) denotes the subalgebra of \(A\) generated by \(e_1,e_2,e_3\) and \(ac\), then \(\quotient{A}{B}\) is self-injective and hence the extension is controllable. This example first appeared in \cite[Example~3.5]{IM25}; it now follows directly from the above corollary.

    \item Frobenius extensions, by \cite[Corollary~8]{Hir59}, have equal classes of relative injective and relative projective modules, therefore if \(AJ(B) \lhd A\) and \(A/J(A) \cong B/J(B)\) as \(\enveloping{B}\)-modules, then they are controllable.
  \end{enumerate}
\end{example}

\begin{rmk}
  The underlying condition \(A/J(A) \cong B/J(B)\) is necessary for the above results, as shown by the examples of Section \ref{sec:NonControllableExtensions}.
\end{rmk}

\subsubsection{Combinatorial constructions}

In \cite[Theorem 3.2]{IM25}, the authors obtained an upper bound for certain extensions of algebras in terms of their associated quivers. More precisely, let \(Q\) be a quiver and let \(Q_0=V_1\sqcup V_2\sqcup\cdots\sqcup V_n\)
be a partition of its vertex set such that
  \begin{enumerate}
    \item if \(i<j\), then there are no arrows from a vertex in \(V_i\) to a vertex in \(V_j\),
    \item for each \(i\), there are no arrows between distinct vertices in \(V_i\).
  \end{enumerate}
Let \(A = \Bbbk Q/I\), with \(I\) any admissible ideal, and consider a subalgebra \(B = \Bbbk R/J\), with \(J\) admissible, such that
  \begin{enumerate}
    \item each idempotent associated with a vertex of \(R_0\) is a sum of idempotents associated to vertices of \(Q_0\),
    \item the arrows \(\beta \in R_1\) are linear combinations of paths in \(Q\) such that for each vertex \(e \in Q_0\), there are vertices \(h, g \in Q_0\) with \(\beta e = h \beta e\) and \(e \beta = e \beta g\),
    \item for any vertex \(e \in Q_0\) and loop \(\gamma \in Q_1\) at \(e\), there is \(\beta \in B\) such that \(\gamma = e\beta\).
  \end{enumerate}
Then the relative global dimension satisfies \(\gldim(A, B) \leqslant n-1\).

\begin{corollary}\label{cor:ProjBoundedAndAlgebraicCondition}
  Let \(B \subseteq A\) be an extension satisfying the above conditions such that \(AJ(B) \lhd A\). Then
    \[
      \gldim\left( \quotient{A}{B} \right) \leqslant \gldim(A, B) \leqslant n-1.
    \]
  In particular, if \(\gldim(\quotient{A}{B}) = n-1\), then \(B \subseteq A\) is controllable.
\end{corollary}

Applying \(-\otimes_A M\) to a relative projective resolution of \(A\) by
\((A^e,B\otimes A^{op})\)-projective modules, we get
\begin{equation} \label{eq:GlrelPDrel}
\gldim(A,B)\leq \pd_{(A^e,B\otimes A^{op})}A
=\pd_{(A^e,B^e)}A,
\end{equation}
where the last equality follows from the commentary of \cite[Page 66]{CLMS20b}. Combining this with \cite[Corollary~3.3]{IM25} yields the following.

\begin{proposition}\label{prop:CombinatorialGldimPdRegularBimodule}
Let \(A=\Bbbk Q/I\) be a finite-dimensional algebra such that the vertices
of \(Q\) can be labeled by \(1,\ldots,n\) in such a way that there are no
arrows \(i\to j\) whenever \(i<j\). Let \(B\subseteq A\), write \(B = \Bbbk R/J\), with \(J\) admissible, satisfying:
\begin{enumerate}
\item each vertex of \(R_0\) is a sum of vertices of \(Q_0\);
\item each arrow \(\beta \in R_1\) satisfies \(\beta=e\beta f\) for some vertices \(e,f\in Q_0\), that is, each \(\beta\) is a uniform element in \(A\), see \cite[Definition 2.9]{Gre99};
\item \(R_1\) contains all loops of \(Q_1\).
\end{enumerate}
If \(\gldim(A,B)=n-1\), then
\[
\gldim(A,B)=\pd_{(A^e,B^e)}A.
\]
\end{proposition}

\begin{proof}
Since \(A\) is weakly triangular, the support of the regular bimodule
\({}_AA_A\) is contained in
\[
\{(i,j)\in Q_0\times Q_0\mid i-j\ge 0\},
\]
using the notation of \cite[Chapter~2]{Ska11}. Set \(n=|Q_0|\). Then
\[
V_r=\{(i,j)\in Q_0\times Q_0\mid i-j=r\},
\qquad 0\le r\le n-1,
\]
defines a partition of the support of \({}_AA_A\). Hence, by
\cite[Theorem~3.2]{IM25},
\[
\pd_{(A^e,B^e)}A\le n-1.
\]
Since \(n-1=\gldim(A,B)\), the result follows \eqref{eq:GlrelPDrel}.
\end{proof}

\begin{example}\label{ex:ProjBoundedArbitraryRelativeGlobalDimension}
  Consider \(A=\Bbbk\mathbb{A}_n\), where
  \[
    \begin{tikzcd}
      {\mathbb{A}_n: 1} \arrow[r, "\alpha_1"] &
      {2} \arrow[r, "\alpha_2"] &
      {\cdots} \arrow[r, "\alpha_{n-1}"] &
      {n}.
    \end{tikzcd}
  \]
  Let \(B\subseteq A\) be the subalgebra generated by all vertices and all paths of even length. The ordinary quivers of these algebras are illustrated below in the case \(n=6\):
  
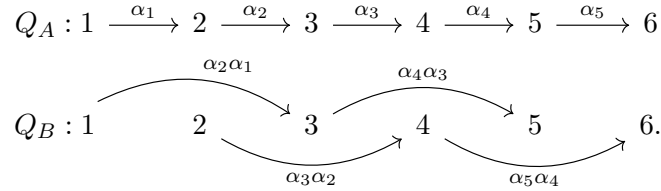
\[
    \begin{tikzcd}
      {Q_A: 1} \arrow[r, "\alpha_1"] &
      {2} \arrow[r, "\alpha_2"] &
      {3} \arrow[r, "\alpha_3"] &
      {4} \arrow[r, "\alpha_4"] &
      {5} \arrow[r, "\alpha_5"] &
      {6} \\
      {Q_B: 1} \arrow[rr, "\alpha_2\alpha_1", bend left] &
      {2} \arrow[rr, "\alpha_3\alpha_2"', bend right] &
      {3} \arrow[rr, "\alpha_4\alpha_3", bend left] &
      {4} \arrow[rr, "\alpha_5\alpha_4"', bend right] &
      {5} &
      {6}.
    \end{tikzcd}
  \]
    In this case, \(J(B)\) consists precisely of all paths in \(\mathbb{A}_n\) of even length. In particular, \(J(B)\) is neither a left nor a right ideal of \(A\), whereas
  \[AJ(B)=J(B)A=J^2(A)=J^2(\Bbbk\mathbb{A}_n) \]
  is a two-sided ideal. Consequently, by Corollary~\ref{cor:ProjBoundedAndAlgebraicCondition} and Proposition~\ref{prop:CombinatorialGldimPdRegularBimodule},
  \[\gldim(\quotient{A}{B})=\gldim(A,B)=\pd_{(A^e,B^e)}A=n-1.\]
\end{example}

\subsubsection{Trivially Twisted Extensions} \label{ex:TriviallyTwistedExtensions}

In \cite{Guo18}, the problem of constructing non-trivial extensions with prescribed relative global dimension was proposed. Partial results were obtained by constructing explicit extensions whose relative global dimension is bounded below by a given non-negative integer, see \cite[Page~2090 and Corollary~1.2]{Guo18}. We briefly recall the notion of a \emph{trivially twisted extension} from the quiver perspective.

Let \(B=\Bbbk Q/I\) and \(C=\Bbbk R/J\) be admissible quotients of path algebras with \(Q_0=R_0\). The \emph{trivially twisted extension} of \(B\) and \(C\) over the semisimple algebra \(\Bbbk Q_0\) is the algebra \(A=\Bbbk S/L\), where
\begin{enumerate}
  \item \(S_0=Q_0\);
  \item \(S_1\) is the disjoint union of \(Q_1\) and \(R_1\); and
  \item \(L\) is generated by \(I\), \(J\), and all paths \(cb\), where \(c\in R_1\) and \(b\in Q_1\).
\end{enumerate}

By \cite[Corollary~1.2]{Guo18},
\[
\gldim(A,B)\geq \gldim(C),
\]
and \cite[Examples~1 and~2]{Guo18} show that \(\gldim(C)\) may be chosen arbitrarily.

On the other hand, the extension \(C\subseteq A\) satisfies the hypotheses of Corollary~\ref{cor:CategoricalControllableInequality}, since \(AJ(C)\lhd A\). Moreover,
\(\quotient{A}{C} \cong B\), and hence
\[ \gldim(A,C)\geq \gldim(B).\]
By arguments similar to those in \cite[Examples~1 and~2]{Guo18}, the algebra \(B\) can be chosen with arbitrary global dimension.

\begin{rmk}
The approach developed in this section yields lower bounds for the global dimension of extensions not covered by \cite[Theorem~1.1]{Guo18}; for instance, see the specialization \(n=2\) in Example~\ref{ex:ProjBoundedArbitraryRelativeGlobalDimension}.
\end{rmk}
        
\subsection{Sufficient Condition for an Extension to be Controllable}

In this subsection we show that the condition \(J(B)\lhd A\) is sufficient for the extension \(B\subseteq A\) to be controllable, in this case \(\quotient{A}{B}\) as an \(A\)-module is \((A, B)\)-projective by Remark \ref{rmk:QuocientCondition}. The proof is based on \cite[Proposition 2.19]{XX13} and uses \cite[Lemma 2.3]{XX13}, which asserts that for every \(M\in\catmod{A}\) there is a split exact sequence of \(B\)-modules
\[
\begin{tikzcd}
0 \arrow[r] & {}_BM \arrow[r, "f"] & {}_BA\otimes_B M \arrow[r, "g"] & {}_B(A/B)\otimes_B M \arrow[r] & 0,
\end{tikzcd}
\]
where \(f(m)=1 \otimes m\) and \(g(a \otimes m)=\pi(a)\otimes m\). For later use, write \(K_M\) for the kernel of the multiplication map \(\mu_M:A\otimes_B M\to M, \)
and set \( \widebar{M}\doteq M/J(B)M. \)

\begin{proposition}\label{prop:KernelsIso} Let \(B \subseteq A\) be an extension of \(\Bbbk\)-algebras such that \(J(B) \lhd A\) and \(M \in \catmod{A}\). Then \(K_M \cong K_{\widebar{M}}\), as \(A\)-modules. \end{proposition}

\begin{proof}
  Consider the exact sequence
  \[
    0 \longrightarrow J(B)M \longrightarrow M \longrightarrow \widebar{M} \longrightarrow 0.
  \]
  Applying \(A \otimes_B -\) and comparing the induced multiplication maps yields the following (solid) commutative diagram with exact rows:
  \[
    \begin{tikzcd}
      {} & 0 \arrow[d] & 0 \arrow[d] & 0 \arrow[d] & {} \\
      {} & K_{J(B)M} \arrow[d] & K_M \arrow[d] \arrow[r, dashed, "\varphi"] & K_{\widebar{M}} \arrow[d] & {} \\
      {} & A\otimes_B J(B)M \arrow[r] \arrow[d] & A\otimes_B M \arrow[r] \arrow[d] & A\otimes_B \widebar{M} \arrow[r] \arrow[d] & 0 \\
      0 \arrow[r] & J(B)M \arrow[r] \arrow[d] & M \arrow[r] \arrow[d] & \widebar{M} \arrow[r] \arrow[d] & 0 \\
      {} & 0 & 0 & 0 &
    \end{tikzcd}
  \]
  where the vertical sequences are induced by the multiplication maps. Since the map\(A\otimes_B J(B)M \longrightarrow J(B)M\)
  is surjective, the Snake Lemma yields a surjective (dashed) homomorphism \(\varphi: K_M \longrightarrow K_{\widebar{M}}. \)

  On the other hand,
  \[ {}_B K_M \cong {}_B(A/B)\otimes_B M \cong {}_B(A/B)\otimes_B \widebar{M} \cong {}_B K_{\widebar{M}},
  \]
  where the middle isomorphism follows from the fact that \(J(B)\lhd A\), which implies that \((A/B)_B\) is semisimple. Hence \({}_B K_M\) and \({}_B K_{\widebar{M}}\) are isomorphic semisimple \(B\)-modules. Viewing \(\varphi\) as a homomorphism of \(B\)-modules, it is a surjection between isomorphic semisimple modules, and therefore an isomorphism. Consequently, \(\varphi\) is an isomorphism of \(A\)-modules.
\end{proof}

\begin{proposition}\label{prop:LocalEqualityPDControllable}
  Let \(B\subseteq A\) be an extension of \(\Bbbk\)-algebras such that
  \(J(B)\lhd A\), and let \(M\in\catmod A\). Suppose that
  \(\pd_{(A,B)}M>1\). Then
  \[
    \pd_{(A,B)}M=\pd_{\quotient{A}{B}}\widebar M.
  \]
\end{proposition}

\begin{proof}
  Since \(\pd_{(A,B)}M>1\), the module \(K_M\) is not \((A,B)\)-projective. By Proposition~\ref{prop:KernelsIso} \(K_M\cong K_{\widebar M},\) hence \(K_{\widebar M}\) is not \(\quotient{A}{B}\)-projective, and therefore \(\pd_{\quotient{A}{B}}\widebar M>1\).
  Consequently,
  \begin{align*}
    \pd_{(A,B)}M
      &=1+\pd_{(A,B)}K_M\\
      &=1+\pd_{(A,B)}K_{\widebar M}\\
      &=\pd_{(A,B)}\widebar M\\
      &=\pd_{\quotient{A}{B}}\widebar M,
  \end{align*}
  where the last equality follows from
  Corollary~\ref{cor:CategoricalControllableInequality}.
\end{proof}

\begin{theorem} \label{teo:BilateralIdealIsControllable} Let \(B \subseteq A\) be an extension of \(\Bbbk\)-algebras such that \(J(B) \lhd A\). Then \(B \subseteq A\) is controllable. \end{theorem}
\begin{proof}
  Since \(J(B)\lhd A\), the ideal \(AJ(B)=J(B)\) is two-sided, and Corollary~\ref{cor:CategoricalControllableInequality} yields
  \[ \gldim(\quotient{A}{B})\le \gldim(A,B).\]
  Thus it remains to prove the reverse inequality.
  If \(\quotient{A}{B}\) is semisimple, then \(J(B)=J(A)\), and
  \cite[Lemma~2.12]{XX13} implies that
  \(\gldim(A,B)=0=\gldim(\quotient{A}{B})\).
  
  Assume now that \(\quotient{A}{B}\) is not semisimple.  If \(\pd_{(A^e,B^e)}A>1\), then Proposition~\ref{prop:LocalEqualityPDControllable},
  applied to the extension \(B^e\subseteq A^e\), gives
  \[
    \pd_{(A^e,B^e)}A=\pd_{({\quotient{A}{B}})^e}\quotient{A}{B}.
  \]
  By \cite[Theorem II]{Eil54},
  \[
    \pd_{(\quotient{A}{B})^e}\quotient{A}{B}=\gldim(\quotient{A}{B}),
  \]
  and therefore
  \[
    \gldim(A,B)\le \pd_{(A^e,B^e)}A=\gldim(\quotient{A}{B}).
  \]

  Finally, if \(\pd_{(A^e,B^e)}A=1\), then
  \(\pd_{(\quotient{A}{B})^e}\quotient{A}{B} \le 1\). Since \(\quotient{A}{B}\) is not
  semisimple, \(\pd_{(\quotient{A}{B})^e}\quotient{A}{B} \neq 0\), hence
  \[
    \pd_{(\quotient{A}{B})^e}\quotient{A}{B}=1=\pd_{(A^e,B^e)}A,
  \]
  and the same argument yields
  \[
    \gldim(A,B)\le \gldim(\quotient{A}{B}).
  \]
  Hence
  \[
    \gldim(A,B)=\gldim(\quotient{A}{B}),
  \]
  proving that \(B\subseteq A\) is controllable.
\end{proof}

\begin{corollary} \label{cor:BilateralIdealTensorProducts}
  Let \(B \subseteq A\) and \(D \subseteq C\) be two extensions such that \(J(B) \lhd A\) and \(J(D) \lhd C\), then
    \begin{enumerate}
      \item[(1)] \(\gldim(A^{e}, B^{e}) = 2 \gldim(A, B)\).
      \item[(2)] \(\gldim(A\otimes C, B \otimes D) = \gldim(A, B) + \gldim(C, D)\).
    \end{enumerate}
\end{corollary}

\begin{proof}
  Statement~(1) follows from~(2) by taking
  \(C=A^{op}\) and \(D=B^{op}\). Thus it suffices to prove~(2).

  Since
  \(
    J(B\otimes D)
    =
    J(B)\otimes D
    +
    B\otimes J(D),
  \)
  and both \(J(B)\) and \(J(D)\) are two-sided ideals of \(A\) and \(C\),
  respectively, it follows that \(J(B\otimes D)\) is a two-sided
  ideal of \(A\otimes C\). Hence, by
  Theorem~\ref{teo:BilateralIdealIsControllable},
  \begin{align*}
    \gldim(A\otimes C,B\otimes D)
      &=
      \gldim\left(\quotient{(A \otimes C)}{(B \otimes D)}\right)\\
      &=
      \gldim\left((\quotient{A}{B})\otimes 
                  (\quotient{C}{D})\right)\\
      &=
      \gldim\left(\quotient{A}{B}\right)
      +
      \gldim\left(\quotient{C}{D}\right)\\
      &=
      \gldim(A,B)+\gldim(C,D),
  \end{align*}
  where the third equality follows from the classical tensor product
  formula for global dimensions.
\end{proof}

\begin{example}
The condition \(J(B)\lhd A\) allows the construction of an extension associated with a given algebra such that the latter controls the corresponding relative homological dimensions. Let \(C\) be a finite-dimensional \(\Bbbk\)-algebra and write
\[
C=\Bbbk Q/I,
\]
where \(I\) is an admissible ideal. Denote by \(J\lhd\Bbbk Q\) the ideal generated by the arrows of \(Q\). Then there exists \(n\ge2\) such that \(
J^n\subseteq I\subseteq J^2. \)
Let \(D\subseteq\Bbbk Q\) be the subalgebra generated by \(I\) and the vertices of \(Q\), and set
\[
A=\frac{\Bbbk Q}{J^n},
\qquad
B=\frac{D}{J^n}.
\]
Then
\[
J(B)=\frac{I}{J^n}\lhd A,
\]
implying that \(B\subseteq A\) is controllable. Moreover, from \(\quotient{A}{B} \cong C\) it follows
\[
\gldim(A,B)
=
\gldim\left(\quotient{A}{B}\right)
=
\gldim(C).
\]
Thus every finite-dimensional algebra arises as the quotient \(\quotient{A}{B}\) of a controllable extension \(B\subseteq A\), and its global dimension coincides with the corresponding relative global dimension.
\end{example}

\section{Tensor Product of Extensions} \label{sec:TensorProductOfExtensions}

To construct extensions with prescribed relative global dimension, we study how this invariant behaves under tensor products of extensions $B \subseteq A$ and $D \subseteq C$. Motivated by the equality $\gldim(A \otimes C) = \gldim(A) + \gldim(C)$ for $B = D = \Bbbk$ \cite[Theorem~16]{Aus55}, we establish certain sufficient conditions for
\[
\gldim(A \otimes C, B \otimes D) = \gldim(A,B) + \gldim(C,D)
\]
to hold. As an application, we show that \(AJ(B)=J(B)A\) is a sufficient condition for a result similar to Corollary \ref{cor:BilateralIdealTensorProducts} to hold.

Throughout this section, we assume that $\operatorname{char}(\Bbbk) \neq 2$.

\subsection{Relative Objects}

\begin{lemma} \label{lemma:TensorOfTwoRelativeProjectivesIsRelativeProjective}
  Let \(B \subseteq A\) and \(D \subseteq C\) be two extensions of finite-dimensional algebras. If \(P \in \mathcal{P}(A,B)\) and \(Q \in \mathcal{P}(C,D)\), then \(P \otimes Q \in \mathcal{P}(A \otimes C,\, B \otimes D).\)
\end{lemma}

\begin{proof}
  By hypothesis, there exist a \(B\)-module \(M\) and a \(D\)-module \(N\) such that \(P\) and \(Q\) are direct summands of \(A \otimes_B M\) and \(C \otimes_D N\), respectively. Hence, \(P \otimes Q\) is a direct summand of
  \[(A \otimes_B M) \otimes (C \otimes_D N)
    \cong
    (A \otimes C) \otimes_{B \otimes D} (M \otimes N).\]
  Therefore,
  \( P \otimes Q \in \mathcal{P}(A \otimes C, B \otimes D). \)
\end{proof}
  
\begin{lemma} \label{lemma:RelativeEnvelopingIsTensorableByAnyModule}
  Let \(B \subseteq A\) be an extension and let \(C\) be a \(\Bbbk\)-algebra. If \(P \in \mathcal{P}(A \otimes C^{\op},\, B \otimes C^{\op})\) and \(N\) is a \(C\)-module, then \(P \otimes_C N \in \mathcal{P}(A,B).\)
\end{lemma}

\begin{proof}
  By definition, the condition
  \(P \in \mathcal{P}(A \otimes C^{\op}, B \otimes C^{\op})\)
  means that \(P\) is a direct summand of a module of the form
  \[
    (A \otimes C^{\op}) \otimes_{B \otimes C^{\op}} M.
  \]
  Consequently, \(P \otimes_C N\) is a direct summand of
  \[
    \bigl((A \otimes C^{\op}) \otimes_{B \otimes C^{\op}} M\bigr)
    \otimes_C N
    \cong
    A \otimes_B (M \otimes_C N),
  \]
  which is induced from a \(B\)-module. Hence,
  \(P \otimes_C N \in \mathcal{P}(A,B)\).
\end{proof}

The next results concern \((A \otimes C, B \otimes D)\)-exact sequences. They rely on the language of double complexes and related results. For the reader's convenience, we briefly recall the necessary notation and facts, and fix our conventions.

Given two complexes:
  \[
    \begin{tikzcd}
      {M_\bullet: \cdots} \arrow[r] & {M_{m}} \arrow[r, "d_m"] & {M_{m-1}} \arrow[r]  & {\cdots} \arrow[r] & {M_2} \arrow[r, "d_2"] & {M_1} \arrow[r, "d_1"] & {M_0} \arrow[r] & {0}
    \end{tikzcd}
  \]
and
  \[
    \begin{tikzcd}
      {N_\bullet: \cdots} \arrow[r] & {N_{n}} \arrow[r, "\delta_n"] & {N_{n-1}} \arrow[r] & {\cdots} \arrow[r] & {N_2} \arrow[r, "\delta_2"] & {N_1} \arrow[r, "\delta_1"] & {N_0} \arrow[r] & {0,}
    \end{tikzcd}
  \]    
Their tensor product, denoted \(M_{\bullet} \otimes N_{\bullet}\), is the double complex with component of bidegree \((m,n)\) given by \(M_m \otimes N_n\). It is customary to depict the tensor product of two complexes by the following diagram:
  \[
    \begin{tikzcd}
      {} & {\vdots} \arrow[d] & {\vdots} \arrow[d] & {\vdots} \arrow[d] & {} \\
      {0} & {M_0 \otimes N_2} \arrow[l] \arrow[d] & {M_1 \otimes N_2} \arrow[d] \arrow[l] & {M_2 \otimes N_2} \arrow[d] \arrow[l] & \cdots \arrow[l] \\
      {0} & {M_0 \otimes N_1} \arrow[l] \arrow[d] & {M_1 \otimes N_1} \arrow[d] \arrow[l] & {M_2 \otimes N_1} \arrow[d] \arrow[l] & {\cdots} \arrow[l] \\
      {0} & {M_0 \otimes N_0} \arrow[l] \arrow[d] & {M_1 \otimes N_0} \arrow[d] \arrow[l] & {M_2 \otimes N_0} \arrow[d] \arrow[l] & {\cdots} \arrow[l] \\
      {} & {0} & {0} & {0} & {}                 
    \end{tikzcd}
  \]
The row differentials are \(D^{\mathrm{r}}_{m,n} = d_m \otimes 1_{N_n}\) and the column differentials are \(D^{\mathrm{c}}_{m,n} = 1_{M_m} \otimes \delta_n\). If the complexes \(M_{\bullet}\) and \(N_{\bullet}\) admit homotopies \(h_i \colon M_i \to M_{i+1}\) and \(\eta_j \colon N_j \to N_{j+1}\), respectively, then each row and each column of the double complex \(M_{\bullet} \otimes N_{\bullet}\) admits a homotopy, given by
  \[
    H^{\mathrm{r}}_{m,n} = h_m \otimes 1_{N_n}
    \quad \text{and} \quad
    H^{\mathrm{c}}_{m,n} = 1_{M_m} \otimes \eta_n,
  \]
respectively.

Recall that the total complex associated to a double complex produces a single complex. For the tensor product of complexes \(M_{\bullet}\) and \(N_{\bullet}\), the homogeneous components of the total complex \(\totalComplex(M_{\bullet} \otimes N_{\bullet})\) are
  \[
    \totalComplex(M_{\bullet} \otimes N_{\bullet})_{s} = \bigoplus_{n+m=s} M_n \otimes N_m,
  \]
with differential \(D = D^{\mathrm{r}} + D^{\mathrm{c}}\), where
  \[
    D^{\mathrm{r}}_{s} = \sum_{n+m=s} d_n \otimes 1_{N_m}, \qquad D^{\mathrm{c}}_{s} = \sum_{n+m=s} (-1)^n 1_{M_n} \otimes \delta_m.
  \]
If the complexes admit homotopies, define
  \begin{align*}
    H^{\mathrm{r}}_{s} &= \sum_{n+m=s} h_n \otimes 1_{N_m}, \\
    H^{\mathrm{c}}_{s} &= \sum_{n+m=s} (-1)^n 1_{M_n} \otimes \eta_m,
  \end{align*}
and set
  \[
    H_s \doteq H^{\mathrm{r}}_{s} + H^{\mathrm{c}}_{s} \colon \totalComplex(M_{\bullet} \otimes N_{\bullet})_{s} \to \totalComplex(M_{\bullet} \otimes N_{\bullet})_{s+1}.
  \]

The next two results construct homotopies for total complexes of tensor products of complexes. This is essential because relative homological objects are computed using exact sequences that admit homotopies.

\begin{lemma} \label{lemma:TotalTensorDoubleComplexIsRelativeExact}
  Let \(B\) and \(C\) be \(\Bbbk\)-algebras. Suppose that
  \[
    \begin{tikzcd}
      {M_\bullet:\; \cdots} \arrow[r] & {M_m} \arrow[r,"d_m"] & {M_{m-1}} \arrow[l,bend left,"h_{m-1}"] \arrow[r] & {\cdots} \arrow[r] & {M_1} \arrow[r,"d_1"] \arrow[l,bend left,"h_1"] & {M_0} \arrow[l,bend left,"h_0"] \arrow[r] & {0}
    \end{tikzcd}
  \]
  and
  \[
    \begin{tikzcd}
      {N_\bullet:\; \cdots} \arrow[r] & {N_n} \arrow[r,"\delta_n"] & {N_{n-1}} \arrow[l,bend left,"\eta_{n-1}"] \arrow[r] & {\cdots} \arrow[r] & {N_1} \arrow[r,"\delta_1"] \arrow[l,bend left,"\eta_1"] & {N_0} \arrow[l,bend left,"\eta_0"] \arrow[r] & {0}
    \end{tikzcd}
  \]
  are \(B\)- and \(C\)-exact, respectively, and admit homotopies \(h_\bullet\) and \(\eta_\bullet\), respectively. Then the total complex
  \(\totalComplex(M_\bullet \otimes N_\bullet)\)
  admits a \(B \otimes C\)-homotopy.
\end{lemma}

\begin{proof}
  A direct approach is to compute \(D_{s+1} H_s + H_{s-1} D_s\) using the decompositions \(D = D^{\operatorname{r}} + D^{\operatorname{c}}\) and \(H = H^{\operatorname{r}} + H^{\operatorname{c}}\). Expanding termwise, the first summand decomposes as
    \begin{align*}
      D_{s+1} \circ H_{s} &= \left( D_{s+1}^{\operatorname{r}} + D_{s+1}^{\operatorname{c}} \right) \circ \left( H_{s}^{\operatorname{r}} + H_{s}^{\operatorname{c}} \right) \\
        &= \left( D_{s+1}^{\operatorname{r}} \circ H_{s}^{\operatorname{r}} \right) + \left( D_{s+1}^{\operatorname{r}} \circ H_{s}^{\operatorname{c}} \right) + \left( D_{s+1}^{\operatorname{c}} \circ H_{s}^{\operatorname{r}} \right) + \Bigl( D_{s+1}^{\operatorname{c}} \circ H_{s}^{\operatorname{c}} \Bigr),
    \end{align*}
  whose computations yield
    \begin{align*}
      D_{s+1}^{\operatorname{r}} \circ H_{s}^{\operatorname{r}} &= \left( \sum_{j=0}^{s} d_{s+1-j} \otimes 1_{N_{j}} \right) \left( \sum_{i=0}^{s} h_{s-i}\otimes 1_{N_{i}} \right) = \sum_{i=0}^{s} d_{s+1-i}h_{s-1} \otimes 1_{N_{i}}, \\
        D_{s+1}^{\operatorname{r}} \circ H_{s}^{\operatorname{c}} &= \left(  \sum_{j=0}^{s} d_{s+1-j} \otimes 1_{N_{j}} \right) \left( \sum_{i=0}^{s} (-1)^{s-i} 1_{M_{s-i}} \otimes \eta_i \right) = \sum_{i=0}^{s-1} (-1)^{s-i}d_{s-i} \otimes \eta_{i}, \\
        D_{s+1}^{\operatorname{c}} \circ H_{s}^{\operatorname{r}} & = \left( \sum_{j=0}^{s} (-1)^{s-j} 1_{M_{s-j}} \otimes \delta_{j+1} \right) \left( \sum_{i=0}^{s} h_{s-i} \otimes 1_{N_{i}} \right) = \sum_{i=0}^{s-1} (-1)^{i+1} h_{i} \otimes \delta_{s-i}, \ \textnormal{and} \\
        D_{s+1}^{\operatorname{c}} \circ H_{s}^{\operatorname{c}} &=  \left( \sum_{j=0}^{s} (-1)^{s-j} 1_{M_{s-j}} \otimes \delta_{j+1} \right) \left( \sum_{i=0}^{s} (-1)^{s-i} 1_{M_{s-i}} \otimes \eta_i \right) = \sum_{i = 0}^{s} 1_{M_{s-i}} \otimes \delta_{i+1}\eta_{i}.
    \end{align*}
  Similarly, the second summand \(H_{s-1}\circ D_{s}\) becomes
    \begin{align*}
      H_{s-1} \circ D_{s} &= \left( H_{s-1}^{\operatorname{r}} + H_{s-1}^{\operatorname{c}} \right) \circ \left( D_{s}^{\operatorname{r}} + D_{s}^{\operatorname{c}} \right) \\
        &= \left( H_{s-1}^{\operatorname{r}} \circ D_{s}^{\operatorname{r}} \right) + \left( H_{s-1}^{\operatorname{r}} \circ D_{s}^{\operatorname{c}} \right) + \left( H_{s-1}^{\operatorname{c}} \circ D_{s}^{\operatorname{r}} \right) + \Bigl( H_{s-1}^{\operatorname{c}} \circ D_{s}^{\operatorname{c}} \Bigr).
    \end{align*}
  Continuing the computation yields the following identities:    
    \begin{align*}
      H_{s-1}^{\operatorname{r}} \circ D_{s}^{\operatorname{r}} &= \left( \sum_{i=0}^{s-1} h_{s-1-i} \otimes 1_{N_{i}} \right) \left( \sum_{j=0}^{s-1} d_{s-j}\otimes 1_{N_{j}} \right) = \sum_{i=0}^{s-1} h_{s-1-i}d_{s-i} \otimes 1_{N_{i}}, \\
        H_{s-1}^{\operatorname{r}} \circ D_{s}^{\operatorname{c}} &= \left( \sum_{i=0}^{s-1} h_{s-1-i} \otimes 1_{N_{i}} \right) \left( \sum_{j=0}^{s-1} (-1)^{s-1-j} 1_{M_{s-1-j}} \otimes \delta_{j+1} \right) = \sum_{i=0}^{s-1} (-1)^{i}h_{i} \otimes \delta_{s-i}, \\
        H_{s-1}^{\operatorname{c}} \circ D_{s}^{\operatorname{r}} & = \left( \sum_{i=0}^{s-1} (-1)^{s-1-i} 1_{M_{s-1-i}} \otimes \eta_{i} \right) \left( \sum_{j=0}^{s-1} d_{s-j}\otimes 1_{N_{j}} \right) \\ &= \sum_{i=0}^{s-1} (-1)^{s-1-i} d_{s-i} \otimes \eta_{i}, \ \textnormal{and} \\
        H_{s-1}^{\operatorname{c}} \circ D_{s}^{\operatorname{c}} &=  \left( \sum_{i=0}^{s-1} (-1)^{s-1-i} 1_{M_{s-1-i}} \otimes \eta_{i} \right) \left( \sum_{j=0}^{s-1} (-1)^{s-1-j} 1_{M_{s-1-j}} \otimes \delta_{j+1} \right) \\ &= \sum_{i = 0}^{s-1} 1_{M_{s-1-i}} \otimes \eta_{i}\delta_{i+1}.
    \end{align*}
  Collecting terms, we obtain
    \begin{align*}
      D^{\operatorname{r}}_{s+1} \circ H^{\operatorname{c}}_{s} + H^{\operatorname{c}}_{s-1}\circ D^{\operatorname{r}}_{s} &= \sum_{i=0}^{s-1} (-1)^{s-i} d_{s-i} \otimes \eta_{i} + \sum_{i=0}^{s-1} (-1)^{s-1-i}d_{s-i} \otimes \eta_i = 0, \\
        D_{s+1}^{\operatorname{c}} \circ H_{s}^{\operatorname{r}} + H_{s-1}^{\operatorname{r}} \circ D_{s}^{\operatorname{c}} &= \sum_{i=0}^{s-1} (-1)^{i+1} h_i \otimes \delta_{s-i} + \sum_{i=0}^{s-1} (-1)^{i}h_{i} \otimes \delta_{s-i} =0,\\
      D_{s+1}^{\operatorname{r}} \circ H_{s}^{\operatorname{r}} + H_{s-1}^{\operatorname{r}} \circ D_{s}^{\operatorname{r}} &= \sum_{i=0}^{s} d_{s+1-i} h_{s-i} \otimes 1_{N_i} + \sum_{i=0}^{s-1} h_{s-1-i} d_{s-i} \otimes 1_{N_i} \\
        &= d_1h_0 \otimes 1_{N_{s}} + \sum_{i=0}^{s-1} (d_{s+1-i} h_{s-i} + h_{s-1-i} d_{s-i}) \otimes 1_{N_{i}} \\
        &= 1_{M_0}\otimes 1_{N_{s}} + \sum_{i=0}^{s-1} 1_{M_{s-i}} \otimes 1_{N_{i}} \\
        &= 1_{\totalComplex(M_{\bullet}\otimes N_{\bullet})_{s}},
      \end{align*}
      \begin{align*}
      D_{s+1}^{\operatorname{c}} \circ H_{s}^{\operatorname{c}} + H_{s-1}^{\operatorname{c}} \circ D_{s}^{\operatorname{c}} &= \sum_{i=0}^{s} 1_{M_{s-i}} \otimes \delta_{i+1}\eta_i + \sum_{i=0}^{s-1} 1_{M_{s-1-i}} \otimes \eta_{i} \delta_{i+1}\\
        &= 1_{M_{s}} \otimes \delta_1 \eta_0 + \sum_{i=1}^{s} 1_{M_{s-i}} \otimes (\delta_{i+1}\eta_i + \eta_{i-1} \delta_{i}) \\
        &= 1_{M_{s}} \otimes 1_{N_0} \sum_{i=1}^{s} 1_{M_{s-i}} \otimes 1_{N_{i}} \\ &= 1_{\totalComplex(M_{\bullet}\otimes N_{\bullet})_{s}}.
    \end{align*}
  Consequently
    \begin{align*}
      D_{s+1} \circ H_{s} + H_{s-1} \circ D_{s} &=\left( D_{s+1}^{\operatorname{r}} \circ H_{s}^{\operatorname{r}} \right) + \left( D_{s+1}^{\operatorname{r}} \circ H_{s}^{\operatorname{c}} \right) + \left( D_{s+1}^{\operatorname{c}} \circ H_{s}^{\operatorname{r}} \right) + \left( D_{s+1}^{\operatorname{c}} \circ H_{s}^{\operatorname{c}} \right) \\
        & \phantom{=} + \left(H_{s-1}^{\operatorname{r}} \circ D_s^{\operatorname{r}}\right) + \left(H_{s-1}^{\operatorname{r}} \circ D_s^{\operatorname{c}}\right) + \left(H_{s-1}^{\operatorname{c}}\circ D_s^{\operatorname{r}}\right) + \left( H_{s-1}^{\operatorname{c}} \circ D_s^{\operatorname{c}} \right) \\
        &= \left( D_{s+1}^{\operatorname{r}} \circ H_{s}^{\operatorname{r}} + H_{s-1}^{\operatorname{r}} \circ D_{s}^{\operatorname{r}} \right) + \left( D_{s+1}^{\operatorname{c}} \circ H_{s}^{\operatorname{c}} + H_{s-1}^{\operatorname{c}} \circ D_{s}^{\operatorname{c}} \right) \\
        &= 1_{\totalComplex(M_{\bullet}\otimes N_{\bullet})_{s}} + 1_{\totalComplex(M_{\bullet}\otimes N_{\bullet})_{s}} = 2 \cdot 1_{\totalComplex(M_{\bullet}\otimes N_{\bullet})_{s}}.
    \end{align*}
 Since \(\operatorname{char}(\Bbbk) \neq 2\), the above equality implies that
    \begin{equation} \label{eq:HalfNaiveHomotopyDoubleComplex}
      \widebar{H_{s}} = \frac{1}{2}\left( H_{s}^{\operatorname{r}} + H_{s}^{\operatorname{c}} \right) = \frac{1}{2}\left( \sum_{s = n + m} h_n \otimes 1_{N_m} + \sum_{s = n+m} (-1)^{n}1_{M_n} \otimes \eta_m \right)
    \end{equation}
  is a \((B\otimes D)\)-homotopy for \(\totalComplex(M_{\bullet} \otimes N_{\bullet})\).
\end{proof}

\begin{proposition}
  Let \(B \subseteq A\) and \(D \subseteq C\) be extensions of algebras. Then the total complex of the tensor product of an \((A,B)\)-exact sequence with a \((C,D)\)-exact sequence is \((A \otimes C, B \otimes D)\)-exact.
\end{proposition}

\begin{proof}
  Follows directly from Lemma \ref{lemma:TotalTensorDoubleComplexIsRelativeExact} and the definition of relative exact sequences.
\end{proof}

It is known that if
  \[
    \begin{tikzcd}
      {M_\bullet: \cdots} \arrow[r] & {P_{m}} \arrow[r, "d_m"] & {P_{m-1}} \arrow[r, "d_{m-1}"]  & {\cdots} \arrow[r] & {P_2} \arrow[r, "d_2"] & {P_1} \arrow[r, "d_1"] & {P_0} \arrow[r, "d_0"] & {M} \arrow[r] & {0}
    \end{tikzcd}
  \]
and
  \[
    \begin{tikzcd}
      {N_\bullet: \cdots} \arrow[r] & {Q_{n}} \arrow[r, "\delta_n"] & {Q_{n-1}} \arrow[r, "\delta_{n-1}"]  & {\cdots} \arrow[r] & {Q_2} \arrow[r, "\delta_2"] & {Q_1} \arrow[r, "\delta_1"] & {Q_0}  \arrow[r, "\delta_0"] & {N} \arrow[r] & {0}
    \end{tikzcd}
  \]
are two acyclic complexes, and \(P_{\bullet}, Q_{\bullet}\) their respective deleted versions, then
  \[
    \begin{tikzcd}
      {\totalComplex(P_{\bullet} \otimes Q_{\bullet})} \arrow[r, "\Delta_0"] & {M \otimes N} \arrow[r] & {0}
    \end{tikzcd}
  \]
is an acyclic complex with \(\Delta_0 = d_0 \otimes \delta_0\). For details, see \cite[Pages 9--10]{CLMS20b}. The following lemma provides another proof of this fact when \(\operatorname{char}(\Bbbk) \neq 2\), adapted to the setting of relative homological algebra by providing a homotopy to the complex above.

\begin{lemma} \label{lemma:HomotopyForDeletedTensorProduct}
  Let \(B\) and \(D\) be \(\Bbbk\)-algebras. Suppose that \(M_\bullet\) is an acyclic complex of \(B\)-modules admitting a \(B\)-homotopy \(h_\bullet\), and let \(P_\bullet\) be its deleted version. Similarly, suppose that \(N_\bullet\) is an acyclic complex of \(D\)-modules admitting a \(D\)-homotopy \(\eta_\bullet\), and let \(Q_\bullet\) be its deleted version. Then
  \[
    \begin{tikzcd}
      {\totalComplex(P_{\bullet} \otimes Q_{\bullet})}
      \arrow[r, "\Delta_0"]
      & {M \otimes N}
      \arrow[r]
      & {0}
    \end{tikzcd}
  \]
  is an acyclic complex of \(B \otimes D\)-modules admitting a \(B \otimes D\)-homotopy.
\end{lemma}

\begin{proof}
  An equivalent way to establish the existence of such a homotopy is to construct homomorphisms of \(B \otimes D\)-modules \(\widetilde{H}_n\) satisfying \[\Delta_{n+1} \widetilde{H}_n \Delta_{n+1} = \Delta_{n+1},\]
where \(\Delta_\bullet\) denotes the differential of
\[
  \begin{tikzcd}
    \totalComplex(P_{\bullet} \otimes Q_{\bullet})
    \arrow[r, "\Delta_0"]
    & M \otimes N
    \arrow[r]
    & 0.
  \end{tikzcd}
\]
  
At degree \(-1\), set
\[
  \widetilde{H}_{-1}
  =
  h_{-1} \otimes \eta_{-1}
  \colon
  M \otimes N \longrightarrow P_0 \otimes Q_0.
\]
Then
\[
  \Delta_0 \widetilde{H}_{-1} \Delta_0
  =
  d_0 h_{-1} d_0 \otimes \delta_0 \eta_{-1} \delta_0
  =
  d_0 \otimes \delta_0
  =
  \Delta_0,
\]
since \(h_\bullet\) and \(\eta_\bullet\) are homotopies for \(d_\bullet\) and \(\delta_\bullet\), respectively.

Notice that the homotopies for \(M_\bullet\) and \(N_\bullet\) yield the decompositions
\[
  P_i = \ker(d_i) \oplus M_i
  \qquad\text{and}\qquad
  Q_j = \ker(\delta_j) \oplus N_j,
\]
where \(M_i\) and \(N_j\) are \(B\)- and \(D\)-submodules, respectively. These decompositions will be used extensively in the constructions that follow.

In degree \(0\), the domain of \(\widetilde{H}_0\) is
\[
  \totalComplex(P_{\bullet} \otimes Q_{\bullet})_0
  =
  P_0 \otimes Q_0.
\]
Using the decompositions above, we obtain
\[
  P_0 \otimes Q_0
  =
  \bigl(\ker(d_0) \otimes \ker(\delta_0)\bigr)
  \oplus
  \bigl(\ker(d_0) \otimes N_0\bigr)
  \oplus
  \bigl(M_0 \otimes \ker(\delta_0)\bigr)
  \oplus
  \bigl(M_0 \otimes N_0\bigr).
\]
The codomain is
\[
  \totalComplex(P_{\bullet} \otimes Q_{\bullet})_1
  =
  (P_1 \otimes Q_0) \oplus (P_0 \otimes Q_1).
\]
 When applied in degree \(1\), the differential \(\Delta\) interacts with the decomposition of \(P_0 \otimes Q_0\) as follows:
\begin{align*}
  (d_1 \otimes 1_{Q_0})(P_1 \otimes Q_0)
  &= \ker(d_0) \otimes Q_0 \\
  &= \bigl(\ker(d_0) \otimes \ker(\delta_0)\bigr)
     \oplus
     \bigl(\ker(d_0) \otimes N_0\bigr), \\
  (1_{P_0} \otimes \delta_1)(P_0 \otimes Q_1)
  &= P_0 \otimes \ker(\delta_0) \\
  &= \bigl(\ker(d_0) \otimes \ker(\delta_0)\bigr)
     \oplus
     \bigl(M_0 \otimes \ker(\delta_0)\bigr).
\end{align*}
Since the submodule \(\ker(d_0) \otimes \ker(\delta_0)\) occurs in the images of both \(P_1 \otimes Q_0\) and \(P_0 \otimes Q_1\), it must be accounted for twice, as was already done in Proposition~\ref{lemma:TotalTensorDoubleComplexIsRelativeExact}. On the other hand, for the summands that occur only once, we can simply apply the naive homotopy. Set
\[
  L =
  \bigl(\ker(d_0) \otimes \ker(\delta_0)\bigr)
  \oplus
  \bigl(M_0 \otimes N_0\bigr).
\]
Define \(\widetilde{H}_0\) on each \(B \otimes D\)-submodule as follows:
\[
  \widetilde{H}_0 =
  \begin{cases}
    \widebar{H}_0,
    & \text{on \(L\), as in Equation~\ref{eq:HalfNaiveHomotopyDoubleComplex},}\\
    h_0 \otimes 1_{Q_0},
    & \text{on \(\ker(d_0) \otimes N_0\),}\\
    1_{P_0} \otimes \eta_0,
    & \text{on \(M_0 \otimes \ker(\delta_0)\).}
  \end{cases}
\]

We now verify the identity
\(\Delta_1 \widetilde{H}_0 \Delta_1 = \Delta_1\)
on \(P_1 \otimes Q_0\). On the summand
\(\ker(d_0) \otimes N_0\), we have
\[
  \Delta_1 \widetilde{H}_0
  =
  d_1 h_0 \otimes 1_{N_0}
  =
  1.
\]
On the summand
\(M_0 \otimes \ker(\delta_0)\), we have
\[
  \Delta_1 \widetilde{H}_0
  =
  \frac{d_1 h_0}{2} \otimes 1_{Q_0}
  +
  1_{P_0} \otimes \frac{\delta_1 \eta_0}{2}
  =
  1.
\]
Thus,
\[
  \Delta_1 \widetilde{H}_0 \Delta_1
  =
  \Delta_1
  \qquad\text{on } P_1 \otimes Q_0.
\]
 Symmetrically, the same argument applies to \(P_0 \otimes Q_1\). Combining the two computations, we obtain
\[
  \Delta_1 \widetilde{H}_0 \Delta_1 = \Delta_1.
\]

For degrees \(n > 0\), we consider
\[
  \begin{tikzcd}
    {\totalComplex(P_\bullet \otimes Q_\bullet)_{n+1}
      = \bigoplus\limits_{i+j=n+1} P_i \otimes Q_j}
    \arrow[rr, "\Delta_{n+1}"', shift right]
    & {}
    &
    {\bigoplus\limits_{r+s=n} P_r \otimes Q_s
      = \totalComplex(P_\bullet \otimes Q_\bullet)_n,}
    \arrow[ll, "\widetilde{H}_n"', shift right]
  \end{tikzcd}
\]
and aim to construct a section \(\widetilde{H}_n\) of the restriction of \(\Delta_{n+1}\) to its image.
  
  To define \(\widetilde{H}_n\), we distinguish two cases according to the indices \(i\) and \(j\). The \emph{frontier} consists of the two direct summands for which \(i=0\) or \(j=0\), while the \emph{interior} consists of all remaining summands. The latter case is simpler, since deleting the original complexes does not affect these summands, and the construction from Lemma~\ref{lemma:TotalTensorDoubleComplexIsRelativeExact} applies directly.

For the interior, define
\[
  \widetilde{H}_n
  =
  \widebar{H}_n,
  \qquad
  \text{on }
  \quad
  \Delta_{n+1}
  \left(
    \bigoplus_{t=1}^{n} P_t \otimes Q_{n+1-t}
  \right),
\]
where \(\widebar{H}_n\) is as in Equation~\ref{eq:HalfNaiveHomotopyDoubleComplex}. With this choice,
\[
  \Delta_{n+1} \widetilde{H}_n \Delta_{n+1}
  =
  \Delta_{n+1}
\]
on
\[
  \bigoplus_{t=1}^{n} P_t \otimes Q_{n+1-t},
\]
since the corresponding identity holds for
\(\widebar{H}_n\) by Equation~\ref{eq:HalfNaiveHomotopyDoubleComplex}.
  
 To construct \(\widetilde{H}_n\) on the frontier, that is, on the images of
\(P_0 \otimes Q_{n+1}\) and \(P_{n+1} \otimes Q_0\) under
\(\Delta_{n+1}\), it suffices, by symmetry, to consider the latter case. The relevant differential images are
\begin{align*}
  (d_1 \otimes 1_{Q_n})(P_1 \otimes Q_n)
  &=
  \ker(d_n) \otimes \ker(\delta_0)
  \oplus
  M_n \otimes \ker(\delta_0), \\
  (1_{P_0} \otimes \delta_{n+1})(P_0 \otimes Q_{n+1})
  &=
  \ker(d_n) \otimes \ker(\delta_0)
  \oplus
  \ker(d_n) \otimes N_0.
\end{align*}

In particular, \(\widetilde{H}_n\) is already defined on
\(\ker(d_n) \otimes \ker(\delta_0)\) by the interior rule. This summand occurs in the image of both \(P_1 \otimes Q_n\) and \(P_0 \otimes Q_{n+1}\) under \(\Delta_{n+1}\). The only remaining summand is
\(\ker(d_n) \otimes N_0\), which occurs in the image of only one of these differential maps. Thus, it suffices to use the naive homotopy and define
\[
  \widetilde{H}_n=h_n \otimes 1_{Q_0} \qquad \text{on } \ker(d_n) \otimes N_0.
\]

    Direct computation gives
    \begin{align*}
      \Delta_{n+1}\widetilde{H}_n&=d_{n+1}h_n \otimes 1_{N_0}=1 \qquad \text{on } \ker(d_n)\otimes N_0, \\
  \Delta_{n+1}\widetilde{H}_n&=\frac{d_{n+1}h_n}{2}\otimes 1_{Q_0}+1_{P_0}\otimes\frac{\delta_1\eta_0}{2}=1 \qquad \text{on } \ker(d_n)\otimes\ker(\delta_0).
\end{align*}
    Hence,
    \[      \Delta_{n+1}\widetilde{H}_n\Delta_{n+1}=\Delta_{n+1}
    \]
    on \(P_{n+1}\otimes Q_0\). By symmetry, the same holds on \(P_0\otimes Q_{n+1}\). This covers all cases and completes the proof.
\end{proof}

\begin{proposition} \label{prop:RelativeResolutionForTensorProduct}
  Let \(B \subseteq A\) and \(D \subseteq C\) two extensions of \(\Bbbk\)-algebras. If
    \[
      P_\bullet \xrightarrow[]{} M \xrightarrow[]{} 0 \ \textnormal{and} \ Q_\bullet \xrightarrow[]{} N \xrightarrow[]{} 0\ 
    \]
  are \((A, B)\) and \((C, D)\)-projective resolutions, respectively, then
    \[
      \begin{tikzcd}
        {\totalComplex(P_{\bullet} \otimes Q_{\bullet})} \arrow[r] & {M\otimes N} \arrow[r] & {0}
      \end{tikzcd}
    \]
  is an \((A \otimes C, B \otimes D)\)-projective resolution.
\end{proposition}

\begin{proof}
    The acyclicity of the above complex and the existence of a \(B \otimes D\)-homotopy follow from Lemma~\ref{lemma:HomotopyForDeletedTensorProduct}. Moreover, by Lemma~\ref{lemma:TensorOfTwoRelativeProjectivesIsRelativeProjective}, \(\totalComplex(P_{\bullet} \otimes Q_{\bullet})\) is a complex of \((A \otimes C, B \otimes D)\)-projective modules.
\end{proof}

\subsection{Relative Dimensions}

There are several statements of the Künneth formulas, in the next result we generalize the following version found in \cite[Chap. XI, Theorem 3.1]{CE99}:
  \[
    \Tor^{A \otimes C}_{n}(X \otimes Y, M \otimes N) \cong \bigoplus_{s+r=n} \Tor_{s}^{A}(X, M) \otimes \Tor_{r}^{C}(Y, N).
  \]

\begin{theorem}[relative Künneth formula] \label{teo:RelativeKunnethFormula}
  Let \(B \subseteq A\) and \(D \subseteq C\) be extensions of \(\Bbbk\)-algebras, and let
\(M \in \catmod{A}\), \(X \in \catmod{A^{\op}}\),
\(N \in \catmod{C}\), and \(Y \in \catmod{C^{\op}}\). Then
    \begin{equation} \label{eq:RelKun}
          \Tor^{(A \otimes C, B \otimes D)}_{ n}(X \otimes Y, M \otimes N) \cong \bigoplus_{s+r=n} \Tor_{s}^{(A, B)}(X, M) \otimes \Tor_{r}^{(C, D)}(Y, N).
    \end{equation}
\end{theorem}

\begin{proof}
  Let \(P_{\bullet} \longrightarrow M \longrightarrow 0\) and \(Q_{\bullet} \longrightarrow N \longrightarrow 0 \) be \((A, B)\) and \((C, D)\)-projective resolutions, respectively. By Proposition \ref{prop:RelativeResolutionForTensorProduct},
    \[
      \begin{tikzcd}
        {\totalComplex(P_{\bullet} \otimes Q_{\bullet})} \arrow[r, "\Delta_0"] & {M\otimes N} \arrow[r] & {0}
      \end{tikzcd}
    \]
  is a \((A \otimes C, B \otimes D)\)-projective resolution for \(M \otimes N\). Recall that
    \[
      \left(X \otimes Y \right) \otimes_{A \otimes C} \totalComplex\left(P_{\bullet} \otimes Q_{\bullet}\right) \cong \totalComplex\left( \left(X \otimes_A M\right) \otimes \left( Y \otimes_C N \right) \right),
    \]
  therefore, the following computes the relative Tor groups:
    \begin{align*}
      \Tor^{(A \otimes C, B \otimes D)}_{n}(X \otimes Y, M \otimes N) &\cong H_n \left( \totalComplex\left( \left(X \otimes_A P_{\bullet} \right) \otimes \left( Y \otimes_C Q_{\bullet} \right) \right) \right)\\
        & \cong (\totalComplex\left( H(X \otimes_A P_{\bullet})\otimes H(Y \otimes_C Q_{\bullet})\right)_n \\
        &\cong \bigoplus_{s+r=n} \Tor_{s}^{(A, B)}(X, M) \otimes \Tor_{r}^{(C, D)}(Y, N),
    \end{align*}
  where the second isomorphism follows from \cite[Chap. IV, Theorem 7.2]{CE99}.
\end{proof}

\begin{rmk}
  In \cite{SO21}, the author develops a relative version of the classical Künneth theorem; see \cite[Theorem 3.44]{SO21}. While that result is expressed in terms of a short exact sequence, it is possible to obtain an isomorphism between the relevant relative Tor groups similar to \eqref{eq:RelKun}. To establish said isomorphism, it is necessary to construct a homotopy for the total complex of the tensor product of two complexes, which is precisely what we do in Proposition~\ref{prop:RelativeResolutionForTensorProduct}.
\end{rmk}

\begin{theorem} \label{teo:ProjDimensionOfTensorOfModules}
  Let \(B \subseteq A\) and \(D \subseteq C\) be two extensions of \(\Bbbk\)-algebras. If \(M \in \catmod{A}\) and \(N \in \catmod{C}\), then
    \[
      \pd_{(A \otimes C, B \otimes D)} M \otimes N = \pd_{(A, B)} M + \pd_{(C, D)}N.
    \]
\end{theorem}

\begin{proof}
  It follows directly from Theorem \ref{teo:RelativeKunnethFormula} and equation \ref{eq:PdAsTorZero}.
\end{proof}

\begin{corollary} \label{cor:GldimTensorProductBoundProjDim}
  Let \(B \subseteq A\) and \(D \subseteq C\) be two extensions, then
    \[
      \gldim(A \otimes C, B \otimes D) \leqq \pd_{(A^{e}, B^{e})} A + \pd_{(C^{e}, D^{e})} C.
    \]
\end{corollary}

\begin{proof}
  Due to the following
    \begin{align*}
      \gldim(A \otimes C, B \otimes D) & \leqq \pd_{\left((A\otimes C)^{e}, B \otimes D \otimes A^{\op} \otimes C^{\op}\right)} A \otimes C \\
        &= \pd_{(A^{e}, B \otimes A^{\op})} A + \pd_{(C^{e}, D \otimes C^{\op})} C, \ \textnormal{by Theorem \ref{teo:ProjDimensionOfTensorOfModules}} \\
        &= \pd_{(A^{e}, B^{e})} A + \pd_{(C^{e}, D^{e})} C.
    \end{align*}
\end{proof}

\begin{definition}\label{def:GeneralizableExtensions}
An extension \(B \subseteq A\) will be called \emph{homologically controllable} if \(AJ(B) \lhd A\), the extension is controllable, and \(\gldim(A,B) = \pd_{(A^{e},B^{e})} A.\)
\end{definition}

\begin{theorem} \label{teo:ComputesGlobalDimensionOfTensorProductOfExtensions}
  Let \(B \subseteq A\) and \(D \subseteq C\) be two homologically controllable extensions. Then \(B \otimes D \subseteq A \otimes C\) is homologically controllable, and
    \[
      \gldim(A \otimes C, B \otimes D) = \gldim(A, B) + \gldim(C, D).
    \]
\end{theorem}

\begin{proof}
  For convenience, write \(F = B \otimes D\) and \(E = A \otimes C\). Since
\(AJ(B) \lhd A\) and \(CJ(D) \lhd C\), we have
\(EJ(F) \lhd E\). In particular, by Corollary \ref{cor:CategoricalControllableInequality},
    \[
      \gldim \left( \quotient{E}{F} \right) \leqq \gldim(E, F).
    \]
  But in this case
    \[
      \quotient{E}{F} = \frac{E}{EJ(F)} \cong \frac{A}{AJ(B)} \otimes \frac{C}{CJ(D)} = (\quotient{A}{B}) \otimes (\quotient{C}{D})
    \]
  and
    \[
      \gldim \left( \quotient{E}{F} \right) = \gldim \left( \quotient{A}{B} \right) + \gldim \left( \quotient{C}{D} \right),
    \]
  by \cite[Theorem 16]{Aus55}. Since both extensions are controllable, the above becomes
    \[
      \gldim(A, B) + \gldim(C, D) \leqq \gldim(E, F).
    \]
  On the other hand, by Corollary~\ref{cor:GldimTensorProductBoundProjDim}
and Theorem~\ref{teo:ProjDimensionOfTensorOfModules},
\begin{align*}
  \gldim(E,F)
  &\leq \pd_{(E^e,F^e)} E \\
  &= \pd_{(A^e,B^e)} A + \pd_{(C^e,D^e)} C \\
  &= \gldim(A,B) + \gldim(C,D),
\end{align*}
where the last equality follows from the hypothesis. Returning to the original notation \(A\otimes C = E\) and \(B\otimes D = F\), we obtain
\[
  \gldim(A\otimes C,B\otimes D)
  =
  \pd_{((A\otimes C)^e,(B\otimes D)^e)}(A\otimes C)
  =
  \gldim(A,B)+\gldim(C,D).
\]
\end{proof}

\begin{example}
  Any extension \(B \subseteq A\) such that \(J(B) \lhd A\) is homologically controllable by Theorem~\ref{teo:BilateralIdealIsControllable} and its proof, regardless of the characteristic of \(\Bbbk\).
\end{example}

\begin{example}
  The extensions \(A \subseteq A\) and \(\Bbbk \subseteq A\) are homologically controllable. In particular, let \(B\) be a non-semisimple algebra and let \(A\) be arbitrary. Applying Theorem~\ref{teo:ComputesGlobalDimensionOfTensorProductOfExtensions} to the extensions
  \[
    \Bbbk \subseteq A
    \qquad\text{and}\qquad
    B \subseteq B,
  \]
  we obtain
  \[
    \gldim(A \otimes B,B)=\gldim(A).
  \]
  Thus, \(B \subseteq A \otimes B\) is a nontrivial extension, and \(A\) can be chosen so that \(\gldim(A)=n \) for any non-negative integer \(n\).
\end{example}

\begin{lemma}\label{lemma:pdBimoduleIsSymmetrical}
  Let \(B \subseteq A\) be an extension of \(\Bbbk\)-algebras. Then
  \[
    \pd_{(A^{e}, B^{e})} A
    =
    \pd_{((A^{\op})^{e}, (B^{\op})^{e})} A^{\op}.
  \]
\end{lemma}

\begin{proof}
  Consider the twist map
  \[
    \tau \colon
    A^{e}=A\otimes_{\Bbbk}A^{\op}
    \longrightarrow
    A^{\op}\otimes_{\Bbbk}A
    \cong
    (A^{\op})^{e},
    \qquad
    a\otimes a' \longmapsto a'\otimes a.
  \]
  This is an isomorphism of \(\Bbbk\)-algebras that restricts to an isomorphism
  \(B^{e} \cong (B^{\op})^{e}.
  \)
  Consequently, the induced equivalence
  \[
    \catmod{A^{e}}
    \simeq
    \catmod{(A^{\op})^{e}}
  \]
  identifies relative projective modules and preserves relative projective dimensions. Since the \(A^{e}\)-module \(A\) is sent to the \((A^{\op})^{e}\)-module \(A^{\op}\), the desired equality follows.
\end{proof}

\begin{corollary} \label{cor:GldimEnveloping2timesGldimExtension}
  Let \(B \subseteq A\) be a controllable extension of \(\Bbbk\)-algebras such that
  \[
    \gldim(A,B)
    =
    \pd_{(A^e,B^e)} A
  \]
  and \(AJ(B)=J(B)A\). Then
  \[
    \gldim(A^e,B^e)=2\gldim(A,B).
  \]
\end{corollary}

\begin{proof}
  The equality \(AJ(B)=J(B)A\) implies that \(AJ(B)\lhd A\) and
  \(A^{\op}J(B^{\op})\lhd A^{\op}\). Together with the hypothesis
  \[
    \gldim(A,B)=\pd_{(A^e,B^e)}A
  \]
  and Lemma~\ref{lemma:pdBimoduleIsSymmetrical}, this shows that both
  \(B\subseteq A\) and \(B^{\op}\subseteq A^{\op}\) are homologically controllable. The result now follows from Theorem~\ref{teo:ComputesGlobalDimensionOfTensorProductOfExtensions} and Equation~\ref{eq:GldimOpposite}.
\end{proof}

\begin{example}
  In the extensions from Example~\ref{ex:ProjBoundedArbitraryRelativeGlobalDimension}, \(J(B)\) is neither a left nor a right ideal of \(A\), although
  \[
    AJ(B)=J(B)A.
  \]
  Consequently, Corollary~\ref{cor:GldimEnveloping2timesGldimExtension} applies and yields examples of homologically controllable extensions satisfying
  \[
    \gldim(A^e,B^e)=2n
  \]
  for every \(n\in\mathbb{N}\).
\end{example}
\section{Non Controllable Extensions} \label{sec:NonControllableExtensions}

The final topic of this work is the study of extensions satisfying the controllable inequality that are not controllable. The constructions presented below are based on the following example.

\begin{example}\label{ex:NonControllableExtension}
  Let \(A\) and \(B\) be the path algebras of the quivers
    \[    
\begin{tikzpicture}[x=.85cm,y=.85cm]

  \node (Q) at (-0.5,0) {\(Q:\)};
  \node (1) at (0,0) {\(1\)};
  \node (2) at (\triangleHeight,\triangleSize/2) {\(2\)};
  \node (3) at (\triangleHeight,-\triangleSize/2) {\(3\)};
  \draw[->] (1) -- (2)
    node[midway, above] {\footnotesize \(\alpha\)};
  \draw[->] (1) -- (3)
    node[midway, below] {\footnotesize \(\beta\)};

  \node (and) at (1.5*\triangleHeight,0) {and};
  \node (R) at (2*\triangleHeight,0) {\(R:\)};
  \node (B1) at (2*\triangleHeight + 0.5,0) {\(1\)};
  \node (B2) at (2*\triangleHeight + \triangleHeight + 0.5,0) {\(2+3,\)};

  \draw[->] (B1) -- (B2)
    node[midway, above] {\footnotesize \(\alpha+\beta\)};

\end{tikzpicture}
    \]
  respectively. Then \(AJ(B)=J(A) \lhd A\), \(\quotient{A}{B}\) is semisimple, and
\[
  \gldim(A, B) = \pd_{(A, B)} M = 1,
\]
where \(M\) is the indecomposable \(A\)-module supported on the vertices \(1\) and \(2\); see \cite[Example 5.5.3]{Pri23} for details.
\end{example}

\begin{rmk}
  By Corollary \ref{cor:CategoricalControllableInequality}, we have, in the above example, \(\pd_{(A, B)} S = 0\) for all simple \(A\)-modules \(S\). This provides a instance of an extension whose global dimension is not determined by simple modules.
\end{rmk}

In light of the above example, it is natural to study how large the difference
\[
  N = \gldim(A, B) - \gldim \left( \quotient{A}{B} \right)
\]
can be. It turns out that \(N\) can be any non-negative integer. This can be described in terms of the ordinary quivers of the algebras in the extensions. In Figure \ref{figure:ArbitraryGapConstruction}, the construction exhibits a cyclic pattern, where “opening” and “closing” of triangles, in a zigzag, are applied sequentially starting from Example \ref{ex:NonControllableExtension}. Regarding relations, all algebras are radical square zero. For instance, in Figure \ref{subfigure:GapOfOne} no relations are imposed, while in  Figure \ref{subfigure:GapOfTwo} the relations \(\gamma\alpha\) and \(\delta\beta\), as well as \((\gamma + \delta)(\alpha + \beta)\), are imposed for \(A_2\) and \(B_2\), respectively. Each step of this construction increases the relative global dimension by one compared to the previous extension, and since each quotient \(\quotient{A}{B}\) is semisimple, this increase is reflected in the gap \(N\).
  \begin{figure}[H]
    \centering
      \begin{subfigure}{0.32\textwidth}
        \centering
          \begin{tikzpicture}[scale=0.9, font=\footnotesize]
            \node (A1) at (-.6, 0) {\small \(A_1\hspace{-2pt}:\)};
            \node (11) at (0, 0) {\(1\)};
            \node (21) at (\triangleHeight, .5*\triangleSize) {\(2\)};
            \node (31) at (\triangleHeight, -.5*\triangleSize) {\(3\)};
            \draw[->] (11) -- (21) node[midway, above] {\(\alpha\)};
            \draw[->] (11) -- (31) node[midway, below] {\(\beta\)};
            \node (B1) at (-.6, -1*\triangleSize) {\small \(B_1\hspace{-2pt}:\)};
            \node (1B1) at (0, -1*\triangleSize) {\( 1\)};
            \node (23B1) at (\triangleHeight, -1*\triangleSize) {\(2+3\)};
            \draw[->] (1B1) -- (23B1) node[midway, above] {\(\alpha+\beta\)};
          \end{tikzpicture}
        \caption{\(\gldim(A_1, B_1) = 1\).}
        \label{subfigure:GapOfOne}
      \end{subfigure}\hfill
    \begin{subfigure}{0.32\textwidth}
        \begin{tikzpicture}[scale=0.85, font=\footnotesize]
          \node (A2) at (-.6, 0) {\small \(A_2\hspace{-2pt}:\)};
          \node (12) at (0, 0) {\(1\)};
          \node (22) at (\triangleHeight, .5*\triangleSize) {\(2\)};
          \node (32) at (\triangleHeight, -.5*\triangleSize) {\(3\)};
          \node (42) at (2*\triangleHeight, 0) {\(4\)};
          \draw[->] (12) -- (22) node[midway, above] {\(\alpha\)};
          \draw[->] (12) -- (32) node[midway, below] {\(\beta\)};
          \draw[->] (22) -- (42) node[midway, above] {\(\gamma\)};
          \draw[->] (32) -- (42) node[midway, below] {\(\delta\)};
          \node (B2) at (-.6, -1*\triangleSize) {\small \(B_2\hspace{-2pt}:\)};
          \node (1B2) at (0, -1*\triangleSize) {\(1\)};
          \node (23B2) at (\triangleHeight, -1*\triangleSize) {\(2+3\)};
          \node (4B2) at (2*\triangleHeight, -1*\triangleSize) {\(4\)};
          \draw[->] (1B2) -- (23B2) node[midway, above] {\(\alpha+\beta\)};
          \draw[->] (23B2) -- (4B2) node[midway, above] {\(\gamma+\delta\)};
        \end{tikzpicture}
      \caption{\(\gldim(A_2, B_2) = 2\).}
      \label{subfigure:GapOfTwo}
    \end{subfigure}\hfill
    \begin{subfigure}{0.32\textwidth}
        \begin{tikzpicture}[scale=0.85, font=\footnotesize]
          \node (A3) at (-.6, 0) {\small \(A_3\hspace{-2pt}:\)};
          \node (13) at (0, 0) {\(1\)};
          \node (23) at (\triangleHeight, .5*\triangleSize) {\(2\)};
          \node (33) at (\triangleHeight, -.5*\triangleSize) {\(3\)};
          \node (43) at (2*\triangleHeight, 0) {\(4\)};
          \node (53) at (2*\triangleHeight, \triangleSize) {\(5\)};
          \draw[->] (13) -- (23) node[midway, above] {\(\alpha\)};
          \draw[->] (13) -- (33) node[midway, below] {\(\beta\)};
          \draw[->] (23) -- (43) node[midway, above] {\(\gamma\)};
          \draw[->] (33) -- (43) node[midway, below] {\(\delta\)};
          \draw[->] (23) -- (53) node[midway, above] {\(\varepsilon\)};
          \node (B3) at (-.6, -1*\triangleSize) {\small \(B_3\hspace{-2pt}:\)};
          \node (1B3) at (0, -1*\triangleSize) {\(1\)};
          \node (23B3) at (\triangleHeight, -1*\triangleSize) {\(2+3\)};
          \node (45B3) at (2*\triangleHeight, -1*\triangleSize) {\(4+5\)};
          \draw[->] (1B3) -- (23B3) node[midway, above] {\(\alpha+\beta\)};
          \draw[->] (23B3) -- (45B3) node[midway, above] {\scriptsize \(\gamma \hspace{-2pt} + \hspace{-2pt} \delta \hspace{-2pt} + \hspace{-2pt} \varepsilon\)};
        \end{tikzpicture}
      \caption{\(\gldim(A_3, B_3) = 3\).}
      \label{subfigure:GapOfThree}
    \end{subfigure}
    \bigskip 
    \vspace{.3cm}
    \begin{subfigure}{0.32\textwidth}
        \begin{tikzpicture}[scale=0.5, font=\footnotesize]
          \node (A4) at (-.8, 0) {\small \(A_4\hspace{-2pt}:\)};
          \node (14) at (0, 0) {\(\bullet\)};
          \node (24) at (\triangleHeight, .5*\triangleSize) {\(\bullet\)};
          \node (34) at (\triangleHeight, -.5*\triangleSize) {\(\bullet\)};
          \node (44) at (2*\triangleHeight, 0) {\(\bullet\)};
          \node (54) at (2*\triangleHeight, \triangleSize) {\(\bullet\)};
          \node (64) at (3*\triangleHeight, .5*\triangleSize) {\(\bullet\)};
          \draw[->] (14) -- (24) node[midway, above] {};
          \draw[->] (14) -- (34) node[midway, below] {};
          \draw[->] (24) -- (44) node[midway, below] {};
          \draw[->] (34) -- (44) node[midway, below] {};
          \draw[->] (24) -- (54) node[midway, below] {};
          \draw[->] (44) -- (64) node[midway, below] {};
          \draw[->] (54) -- (64) node[midway, below] {};
          \node (B4) at (-.8, -1*\triangleSize) {\small \(B_4\hspace{-2pt}:\)};
          \node (1B4) at (0, -1*\triangleSize) {\(\bullet\)};
          \node (23B4) at (\triangleHeight, -1*\triangleSize) {\(\bullet\)};
          \node (45B4) at (2*\triangleHeight, -1*\triangleSize) {\(\bullet\)};
          \node (6B4) at (3*\triangleHeight, -1*\triangleSize) {\(\bullet\)};
          \draw[->] (1B4) -- (23B4) node[midway, below] {};
          \draw[->] (23B4) -- (45B4) node[midway, below] {};
          \draw[->] (45B4) -- (6B4) node[midway, below] {};
        \end{tikzpicture}
      \caption{\(\gldim(A_4, B_4) = 4\).}
    \end{subfigure}\hfill
    \begin{subfigure}{0.32\textwidth}
        \begin{tikzpicture}[scale=0.5, font=\footnotesize]
          \node (A5) at (-.8, 0) {\small \(A_5\hspace{-2pt}:\)};
          \node (15) at (0, 0) {\(\bullet\)};
          \node (25) at (\triangleHeight, .5*\triangleSize) {\(\bullet\)};
          \node (35) at (\triangleHeight, -.5*\triangleSize) {\(\bullet\)};
          \node (45) at (2*\triangleHeight, 0) {\(\bullet\)};
          \node (55) at (2*\triangleHeight, \triangleSize) {\(\bullet\)};
          \node (65) at (3*\triangleHeight, .5*\triangleSize) {\(\bullet\)};
          \node (75) at (3*\triangleHeight, -.5*\triangleSize) {\(\bullet\)};
          \draw[->] (15) -- (25) node[midway, above] {};
          \draw[->] (15) -- (35) node[midway, below] {};
          \draw[->] (25) -- (45) node[midway, below] {};
          \draw[->] (35) -- (45) node[midway, below] {};
          \draw[->] (25) -- (55) node[midway, below] {};
          \draw[->] (45) -- (65) node[midway, below] {};
          \draw[->] (55) -- (65) node[midway, below] {};
          \draw[->] (45) -- (75) node[midway, below] {};
          \node (B5) at (-.8, -1*\triangleSize) {\small \(B_5\hspace{-2pt}:\)};
          \node (1B5) at (0, -1*\triangleSize) {\(\bullet\)};
          \node (23B5) at (\triangleHeight, -1*\triangleSize) {\(\bullet\)};
          \node (45B5) at (2*\triangleHeight, -1*\triangleSize) {\(\bullet\)};
          \node (67B5) at (3*\triangleHeight, -1*\triangleSize) {\(\bullet\)};
          \draw[->] (1B5) -- (23B5) node[midway, below] {};
          \draw[->] (23B5) -- (45B5) node[midway, below] {};
          \draw[->] (45B5) -- (67B5) node[midway, below] {};
        \end{tikzpicture}
      \caption{\(\gldim(A_5, B_5) = 5\).}
    \end{subfigure}\hfill
    \begin{subfigure}{0.32\textwidth}
        \begin{tikzpicture}[scale=0.45, font=\scriptsize]
          \node (A6) at (-.9, 0) {\small  \(A_6\hspace{-2pt}:\)};
          \node (16) at (0, 0) {\(\bullet\)};
          \node (26) at (\triangleHeight, .5*\triangleSize) {\(\bullet\)};
          \node (36) at (\triangleHeight, -.5*\triangleSize) {\(\bullet\)};
          \node (46) at (2*\triangleHeight, 0) {\(\bullet\)};
          \node (56) at (2*\triangleHeight, \triangleSize) {\(\bullet\)};
          \node (66) at (3*\triangleHeight, .5*\triangleSize) {\(\bullet\)};
          \node (76) at (3*\triangleHeight, -.5*\triangleSize) {\(\bullet\)};
          \node (86) at (4*\triangleHeight, 0) {\(\bullet\)};
          \draw[->] (16) -- (26) node[midway, above] {};
          \draw[->] (16) -- (36) node[midway, below] {};
          \draw[->] (26) -- (46) node[midway, below] {};
          \draw[->] (36) -- (46) node[midway, below] {};
          \draw[->] (26) -- (56) node[midway, below] {};
          \draw[->] (46) -- (66) node[midway, below] {};
          \draw[->] (56) -- (66) node[midway, below] {};
          \draw[->] (46) -- (76) node[midway, below] {};
          \draw[->] (66) -- (86) node[midway, below] {};
          \draw[->] (76) -- (86) node[midway, below] {};
          \node (B6) at (-.9, -1*\triangleSize) {\small \(B_6\hspace{-2pt}:\)};
          \node (1B6) at (0, -1*\triangleSize) {\(\bullet\)};
          \node (23B6) at (\triangleHeight, -1*\triangleSize) {\(\bullet\)};
          \node (45B6) at (2*\triangleHeight, -1*\triangleSize) {\(\bullet\)};
          \node (67B6) at (3*\triangleHeight, -1*\triangleSize) {\(\bullet\)};
          \node (8B6) at (4*\triangleHeight, -1*\triangleSize) {\(\bullet\)};
          \draw[->] (1B6) -- (23B6) node[midway, below] {};
          \draw[->] (23B6) -- (45B6) node[midway, below] {};
          \draw[->] (45B6) -- (67B6) node[midway, below] {};
          \draw[->] (67B6) -- (8B6) node[midway, below] {};
        \end{tikzpicture}
      \caption{\(\gldim(A_6, B_6) = 6\).}
      \label{subfigure:GapOfSix}
    \end{subfigure}
    \caption{\textit{Examples of gaps from \(1\) to \(6\). The relations of all algebras above are given by the square of the radical of their associated path algebras. Each \(B_n\), viewed as a subalgebra of \(A_n\), is generated by sums of idempotents and arrows in the same vertical alignment, as indicated in the first three examples.}}
    \label{figure:ArbitraryGapConstruction}
  \end{figure}
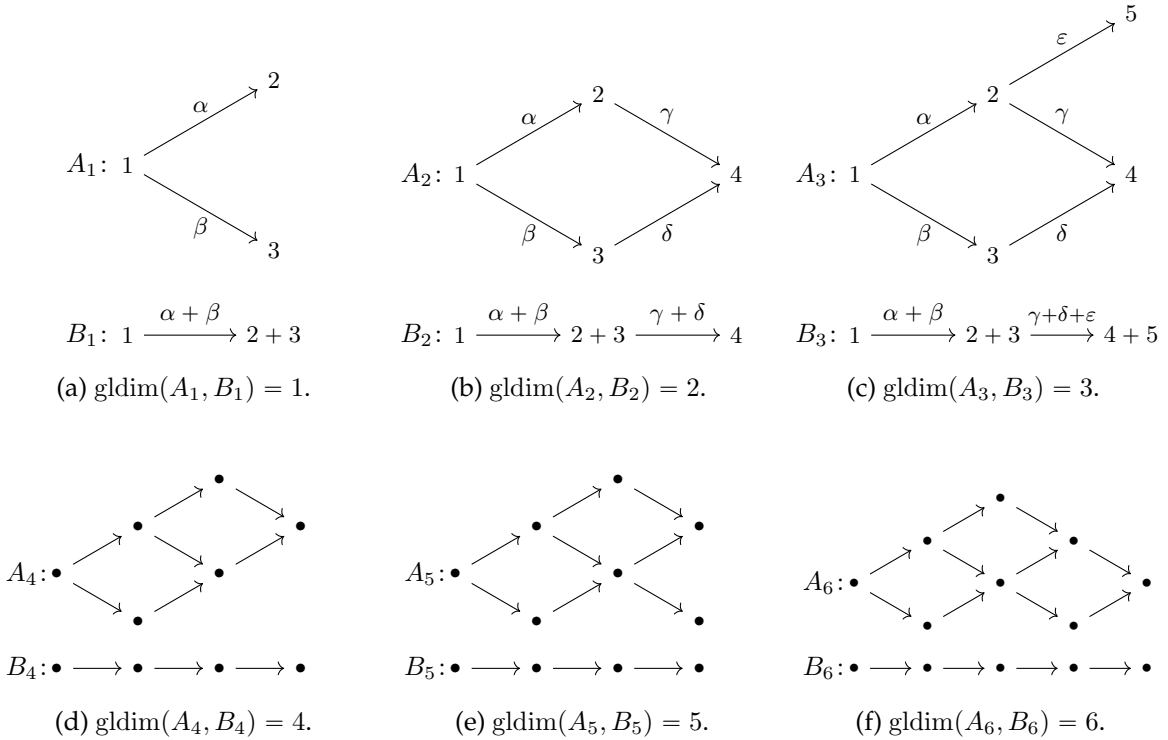

Let \(B_n \subseteq A_n\) denote the \(n^{\text{th}}\)-step extension obtained by the above process. The following results formalize the construction.

\begin{lemma}\label{lemma:RepresentationPropertiesOfGap}
  For the extension \(B_n \subseteq A_n\), the following hold:
  \begin{enumerate}
    \item \(A_n J(B_n) A_n = J(A_n)\). In particular, \(\displaystyle \quotient{A_{n}}{B_{n}}\) is semisimple.
    \item Both \(B_n\) and \(A_n\) are of finite representation type.
  \end{enumerate}
\end{lemma}

\begin{proof}
  For each arrow \(x \in A_n\), there exist idempotents \(e, f \in A_n\) such that \(fJ(A_n)e = \langle x \rangle\). Moreover, \(x\) is a summand of an arrow of \(B_n\), hence
  \[
    x \in \langle x \rangle = fJ(B_n)e \subseteq A_n J(B_n) A_n,
  \]
  which proves the first assertion.

  The second assertion follows from \cite[Theorem 1]{Kru75} or \cite{Gab72}.
\end{proof}

\begin{lemma}\label{lemma:GapRelativeProjectiveModules}
  For the extensions \(B_n \subseteq A_n\), the following hold:
  \begin{enumerate}
    \item The relatively projective simple \(A_n\)-modules are \(S(1)\) and those supported on vertices \(e\) such that \(S(e)=A_n e = P(e)\).
    \item A two-dimensional indecomposable \(A_n\)-module is \((A_n,B_n)\)-projective if and only if it is \(A_n\)-projective.
    \item All three-dimensional indecomposable \(A_n\)-modules are relatively projective.
    \item No four-dimensional indecomposable \(A_n\)-module is relatively projective.
  \end{enumerate}
\end{lemma}

\begin{proof}
  Note that, by the proof of Lemma \ref{lemma:RepresentationPropertiesOfGap}, all possible indecomposable \(A_n\)-modules are analyzed in this lemma.
  
  Basic computations show that \(S(1)= A_n \otimes_{B_n} S(1)\). If \(e \in A_n\) is a vertex such that \(S(e)\) is not \(A_n\)-projective, then \(A_n \otimes_{B_n} S(e)\) is a three-dimensional indecomposable \(A_n\)-module supported on \(e\) and two other vertices \(\{f,g\}\). Restricting to the vertices with non-zero vector spaces, we obtain that
    \setlength\arraycolsep{1pt}
\[
  \begin{tikzpicture}
    \node (P) at (-2,0) {\(A_n\otimes_{B_n} S(e):\)};
    \node (e) at (0, 0.8*\triangleSize/2) {\((\Bbbk)_e\)};
    \node (f) at (0, -0.8*\triangleSize/2) {\((\Bbbk)_f\)};
    \node (g) at (0.8*\triangleHeight, 0) {\((\Bbbk)_g\)};
    \draw[->] (e) -- (g) node[midway, above] {\(\cong\)};
    \draw[->] (f) -- (g) node[midway, below] {\(\cong\)};
  \end{tikzpicture}
\]
which proves the first assertion.

Similarly, if
\[
  \begin{tikzpicture}
    \node (M) at (-1,0) {\(M:\)};
    \node (e) at (0, 0.8*\triangleSize/2) {\((\Bbbk)_e\)};
    \node (f) at (0, -0.8*\triangleSize/2) {\((0)_f\)};
    \node (g) at (0.8*\triangleHeight, 0) {\((\Bbbk)_g\)};
    \draw[->] (e) -- (g) node[midway, above] {\(\cong\)};
  \end{tikzpicture}
\]
is a two-dimensional non-projective \(A_n\)-module, then
\[
A_n \otimes_{B_n} M \cong A_n \otimes_{B_n} S(e)
\]
for some vertex \(e \in A_n\), which proves the second assertion. The third assertion follows from the first two.

Finally, the relative projective cover of an indecomposable four-dimensional \(A_n\)-module
\[
  \begin{tikzpicture}
    \node (M) at (-1, -0.8*\triangleSize/4) {\(M:\)};
    \node (e) at (0, 0.8*\triangleSize/2) {\((\Bbbk)_e\)};
    \node (f) at (0, -0.8*\triangleSize/2) {\((\Bbbk)_f\)};
    \node (g) at (0.8*\triangleHeight, 0) {\((\Bbbk)_g\)};
    \node (h) at (0.8*\triangleHeight, -0.8*\triangleSize) {\((\Bbbk)_h\)};
    \draw[->] (e) -- (g) node[midway, above] {\(\cong\)};
    \draw[->] (f) -- (g) node[midway, below] {\(\cong\)};
    \draw[->] (f) -- (h) node[midway, below] {\(\cong\)};
  \end{tikzpicture}
\]
is \((A_n e \oplus A_n f)^2 = P(e)^2 \oplus P(f)^2\), which proves the fourth assertion.
\end{proof}

\begin{proposition}
 For the extensions \(B_n \subseteq A_n\) constructed above, we have
\[
\gldim(A_n, B_n) = n.
\]
\end{proposition}

\begin{proof}
  The proof consists of a case-by-case analysis of minimal relative projective resolutions of indecomposable \(A_n\)-modules. This is possible since \(A_n\) is of finite representation type and Lemma \ref{lemma:GapRelativeProjectiveModules} classifies the relatively projective indecomposable modules.

  Carrying out the computation, and using the notation of Figure \ref{figure:ArbitraryGapConstruction}, we obtain
  \[
    \pd_{(A_n, B_n)} M \leq
    \pd_{(A_n, B_n)}
    \left(
      \begin{tikzcd}
        (\Bbbk)_1 \arrow[r, "\cong"] & (\Bbbk)_2
      \end{tikzcd}
    \right)
    = n.
  \]

  Hence \(\gldim(A_n, B_n) = n\). Moreover,
  \[
    \gldim(A_n, B_n) - \gldim(\quotient{A_{n}}{B_{n}})
    = n - 0 = n.
  \]
\end{proof}

\begin{rmk}
  It is not known whether the gap can be negative, or whether every positive integer can be realized as the gap of an extension \(B \subseteq A\) satisfying \(AJ(B) \lhd A\).
\end{rmk}

\begin{rmk}
  Consider \(C_n \subseteq A_n\) to be the subalgebra generated by \(\langle \Sigma_{B_{n}}, J(A_n) \rangle\), where \(\Sigma_{B_n}\) is the semisimple algebra generated by the idempotents of \(B_n\). In quiver terms, for \(n=6\), it is the radical square zero algebra given by
   \begin{equation*}
      \begin{tikzcd}
        \bullet \arrow[r, bend left] \arrow[r, bend right] & \bullet \arrow[r] \arrow[r, bend left, shift left] \arrow[r, bend right, shift right] & \bullet \arrow[r] \arrow[r, bend left, shift left] \arrow[r, bend right, shift right] & \bullet \arrow[r] \arrow[r, bend left, shift left] \arrow[r, bend right, shift right] & \bullet \arrow[r, bend left] \arrow[r, bend right] & \bullet
      \end{tikzcd}
    \end{equation*}
  with omitted labels; regarding Figure \ref{subfigure:GapOfSix}, we are only summing the idempotents in the same vertical alignment without summing the arrows. The extension \(C_n \subseteq A_n\) is controllable, for they have the same radical, by Corollary \ref{cor:EqualRadicalExtensionSemisimple}. Moreover, since all involved algebras are radical square zero, then
    \[
      C_{n} J(B_{n}) C_{n} = J(B_{n}),
    \]
  so that \(B_{n} \subseteq C_{n}\) is homologically controllable, by Theorem \ref{teo:BilateralIdealIsControllable}. This means that even though the extension \(B_{n} \subseteq A_{n}\) is non-controllable, it can be viewed as a tower of extensions \(B_{n} \subseteq C_{n} \subseteq A_{n}\) such that the partial extensions, \(B_{n} \subseteq C_{n}\) and \(C_{n} \subseteq A_{n}\), are homologically controllable.
  
  This remark raises the following question: let \(B \subseteq A\) be an extension of algebras that is not controllable, suppose that both algebras are given by quivers and relations, is it possible to construct an algebra \(C\) such that both \(B \subseteq C\) and \(C \subseteq A\) are controllable?
\end{rmk}

\bibliographystyle{alpha}
\bibliography{bib}

\end{document}